\documentclass{amsart}
\usepackage{amsmath, amssymb}
\usepackage{graphicx}
\makeatletter
  \newcommand\figcaption{\def\@captype{figure}\caption}
  \newcommand\tabcaption{\def\@captype{table}\caption}
\makeatother

\newcommand{\p}{\partial}
\numberwithin{equation}{section}
\newtheorem{theorem}{Theorem}[section]
\newtheorem{proposition}{Proposition }[section]

\newtheorem{remark}{Remark}[section]

\begin{document}

\title[Finite element methods for the stochastic mean curvature flow]{Finite 
element approximations of the stochastic mean curvature flow of planar 
curves of graphs}

\author{Xiaobing Feng}
\address{Department of Mathematics \\
         The University of Tennessee \\
         Knoxville, TN 37996, U.S.A.}
\email{xfeng@math.utk.edu}

\author{Yukun Li}
\address{Department of Mathematics \\
         The University of Tennessee \\
         Knoxville, TN 37996, U.S.A.}
\email{yli@math.utk.edu}
\thanks{The work of the first two authors was partially supported by 
the NSF grant DMS-1016173.}

\author{Andreas Prohl}
\address{Mathematisches Institut, Universit\"at T\"ubingen, Auf der 
Morgenstelle 10, D-72076, T\"ubingen, Germany}
\email{ prohl@na.uni-tuebingen.de.}



\begin{abstract}
This paper develops and analyzes a semi-discrete and  a fully discrete 
finite element method for a one-dimensional quasilinear parabolic stochastic 
partial differential equation (SPDE) which describes the stochastic 
mean curvature flow for planar curves of graphs. To circumvent the 
difficulty caused by the low spatial regularity of the SPDE solution,  
a regularization procedure is first proposed to approximate 
the SPDE, and an error estimate for the regularized problem is derived.  
A semi-discrete finite element method, and a space-time fully discrete 
method are then proposed to approximate the solution of the regularized 
SPDE problem.  Strong convergence with rates are established for both, 
semi- and fully discrete methods. Computational experiments are provided to study 
the interplay of the geometric evolution and gradient type-noises.
\end{abstract}

\maketitle
\section{Introduction}\label{sec-1}
The mean curvature flow (MCF) refers to a one-parameter family 
of hypersurfaces $\{\Gamma_t\}_{t\geq 0} \subset \mathbf{R}^{d+1}$ 
which starts from a given initial surface $\Gamma_0$ and evolves 
according to the geometric law
\[
V_n(t, \cdot) = H(t, \cdot),
\]
where $V_n(t, \cdot)$ and $H(t,\cdot)$ denote respectively the 
normal velocity and the mean curvature of the hypersurface 
$\Gamma_t$ at time $t$. The MCF is the best known curvature-driven 
geometric flow which finds many applications in differential geometry,  
geometric measure theory, image processing and materials science and 
have been extensively studied both analytically and numerically 
(cf. \cite{Deckelnick_Dziuk_Elliott05, Giga06,Osher_Fedkiw03,
Sethian99, Zhu02} and the references therein).

As a geometric problem, the MCF can be described using different 
formulations. Among them,  we mention the classical parametric 
formulation \cite{Huisken84},  Brakke's varifold formulation 
\cite{Brakke78}, De Giorgi's barrier function formulation 
\cite{DeGiorgi94, Bellettini_Novaga97, Bellettini_Paolini95}, 
the variational formulation \cite{Almgren_Taylor_Wang93}, the level 
set formulation \cite{Osher_Sethian88, Evans_Spruck91, Chen_Giga_Goto89}, 
and the phase field formulation \cite{Evans_Soner_Souganidis92, Ilmanen93}.  
We remark that different formulations often lead to different solution 
concepts and also lead to developing different analytical (and numerical) 
concepts and techniques to analyze and approximate the MCF.  However,  
all these formulations of the MCF give rise to difficult but 
interesting nonlinear geometric partial differential equations (PDEs),
and the resolution of the MCF then depends on the solutions of 
these nonlinear geometric PDEs.
One interesting feature of the MCF is the development of singularities, 
in particular singularities which may occur in finite time, even when 
the initial hypersurface is smooth.  The singularities may appear in 
different  forms such as self-intersection, pinch-off,  merging, and 
fattening.  To understand and characterize these singularities have been 
the focus of the analytical and numerical research on the MCF 
(cf.~\cite{Chen_Giga_Goto89,Deckelnick_Dziuk_Elliott05,Evans_Spruck91, 
Feng_Prohl03, Osher_Fedkiw03, Sethian99, Zhu02}, and the references therein).

For application problems,  there is a great deal of interest to 
include stochastic effects, and to study the impact of special noises 
on  regularities of solutions, as well as their long-time behaviors.  
The uncertainty may arise from various sources such as thermal fluctuation, 
impurities of the materials, and the intrinsic instabilities of the 
deterministic evolutions.  In this paper we consider the following 
form of a stochastically perturbed mean curvature flow: 
\begin{equation}\label{e1.1}
V_n= H(t,\cdot) + \epsilon \dot{W}_t,
\end{equation}
where $\dot{W}$ denotes a white in time noise, and $\epsilon>0$ is  
a constant.  It is easy to check that (cf. \cite{Yip98,Sarhir_Renesse}) 
the level set formulation of \eqref{e1.1} is given by the following 
nonlinear parabolic stochastic partial differential equation (SPDE):
\begin{equation} \label{e1.2}
df=|\nabla_{x'} f|\,  \mbox{div}_{x'} \Bigl( \frac{\nabla_{x'} f}{|\nabla_{x'} f|} 
\Bigr)\, dt + \epsilon |\nabla_{x'} f| \circ dW_t,
\end{equation}
where $f=f(x',t)$ with $x'=(x, x_{d+1})$ denotes the level set function 
so that $\Gamma_t$ is represented by the zero level set of $f$, 
and  `$\circ$' refers to the Stratonovich interpretation of the 
stochastic integral. Again, stochastic effects are modeled by a standard 
${\mathbb R}$-valued Wiener process $W \equiv \{ W_t;\, t \geq 0\}$ 
which is defined on a given filtered probability space 
$(\Omega, {\mathcal F}, \{ {\mathcal F}_t;\, t \geq 0\}, {\mathbb P})$.

In the case that  $f$ is a $d$-dimensional graph, that is, 
$f(x',t)=x_{d+1}- u(x,t)$, equation \eqref{e1.2} reduces to 
\begin{equation} \label{e1.2a}
du=\sqrt{1+|\nabla_x u|^2}\,  \mbox{div}_{x} 
\Bigl( \frac{\nabla_x u}{\sqrt{1+|\nabla_x u|^2}} \Bigr)\, dt
+ \epsilon \sqrt{1+|\nabla_x u|^2}  \circ dW_t.
\end{equation}
To the best of our knowledge,  a comprehensive PDE theory for the 
SPDE \eqref{e1.2a}  is still missing in the literature. For the 
case $d=1$,   \eqref{e1.2a} reduces to the following one-dimensional 
nonlinear parabolic SPDE: 
\begin{align} \label{e1.3a}
du &= \frac{\p^2_x u}{\sqrt{1+|\p_x u|^2} } dt   
+ \epsilon \sqrt{1+|\p_x u|^2}\circ dW_t \\ \nonumber 
&= \partial_x \bigl( {\rm arctan} (\partial_x u)\bigr) dt 
+ \epsilon \sqrt{1+|\p_x u|^2} \circ dW_t.
\end{align}
Here $\p_x u$ stands for the derivative of $u$ with respect to $x$. This 
Stratonovich SPDE can be equivalently converted into the following It\^o SPDE:
\begin{align} \label{e1.3}
du &= \Bigl[ \frac{\epsilon^2}2 \p^2_x u +\bigl( 1-\frac{\epsilon^2}2\bigr)  
\frac{\p_x^2 u}{\sqrt{1+|\p_x u|^2} } \Bigr] dt  
+ \epsilon \sqrt{1+|\p_x u|^2}\,  dW_t \\ \nonumber
&= \partial_x \Bigl(  \frac{\varepsilon^2}{2} \partial_x u 
+ (1- \frac{\varepsilon^2}{2}){\rm arctan} (\partial_x u)\Bigr) dt
 + \epsilon \sqrt{1+|\p_x u|^2}\,  dW_t.
\end{align}

As is evident from (\ref{e1.3a}), (\ref{e1.3}), the stochastic 
mean curvature flow (\ref{e1.2a}) for $d=1$ may be interpreted
as a gradient flow with multiplicative noise.  Recently,  Es-Sarhir 
and von Renesse \cite{Sarhir_Renesse} proved existence and uniqueness  
of  (stochastically) strong solutions  for \eqref{e1.3a}  by a variational 
method, based on the Lyapunov structure of the problem 
(cf.~\cite[property (H3)]{Sarhir_Renesse}) which replaces the  standard  
coercivity assumption (cf.~\cite[property (A)]{Sarhir_Renesse}).
As is pointed out in \cite{Sarhir_Renesse}, mild solutions for 
\eqref{e1.3a} may not be expected due to its quasilinear character.

The primary goal of this paper is to develop and analyze by a variational 
method some semi-discrete and fully discrete finite element methods for 
approximating (with rates) the strong solution of the It\^o 
form \eqref{e1.3} of the stochastic MCF. The error 
analysis presented in this paper differs from most existing works on the  
numerical analysis of SPDEs, where mild solutions are mostly 
approximated with the help of corresponding discrete 
semi-groups (see ~\cite{KLM1} and the references therein).  
We also note that the error estimates derived in \cite{Gyongy1} 
which hold for general quasilinear SPDEs do not apply to (\ref{e1.3}) 
because the structural assumptions, such as the coercivity 
assumption \cite[cf.~Assumption 2.1, (ii)]{Gyongy1} and the
strong monotonicity assumption \cite[cf.~Assumption 2.2, (i)]{Gyongy1}
fail to hold for \eqref{e1.3}, and also the regularity assumptions 
\cite[cf.~Assumption 2.3]{Gyongy1} are not known to hold in the 
present case.  In this paper, we use a variational approach similar 
to \cite{Gyongy1,BCP1,CP1} to analyze the convergence of our finite 
element methods. One main difficulty for approximating the strong 
solution of \eqref{e1.3} with certain rates is caused by the low regularity
 of the solution.  To circumvent this difficulty,  we first regularize 
the SPDE \eqref{e1.3} by adding an additional linear diffusion
term  $\delta \p_x^2 u$ to the drift coefficient of (\ref{e1.3}); 
as a consequence the related drift operator in \eqref{e2.4} becomes
strongly monotone, and the corresponding solution process
$u^\delta$ is then $H^2$-valued in space. However, it is due to the
`gradient-type' noise that a relevant H\"older estimate in the $H^1$-norm for
the solution $u^\delta$ seems not available, which is necessary to properly control
time-discretization errors.      
In order to circumvent this problematic issue, we proceed first with the
spatial discretization (\ref{eq3.1})-(\ref{eq3.1x}); we may then use an
inverse finite element estimate, and the weaker H\"older estimate
(\ref{prop3.7b}) for the process $u^\delta_h$ to control time-discretization
errors. We remark that addressing space discretization errors
first requires to efficiently cope with the limited regularity of Lagrange
finite element functions in the context of required higher norm estimates,
which is overcome by a perturbation argument (cf.~Proposition~\ref{prop3.7}).

The remainder of this paper consists of three additional sections.  
In section \ref{sec-2} we first recall some relevant facts about the 
solution of \eqref{e1.3} from \cite{Sarhir_Renesse}; we then present 
an analysis for the regularized problem. The main result of this 
section is to prove an error bound for $u^\delta-u$ in powers of 
$\delta$. In section \ref{sec-3} we propose a semi-discrete (in space) 
and a fully discrete finite element method for the regularized 
equation \eqref{e2.4} of the SPDE \eqref{e1.3}. The main result of 
this section is the strong $L^2$-error estimate for the finite element 
solution.  Finally,  in section \ref{sec-4} we present several 
computational results to validate the theoretical error estimate,
and to study relative effects due to geometric evolution and gradient-type noises.

\section{Preliminaries and error estimates for a PDE regularization}
\label{sec-2}
The standard function and space notation will be adapted in this paper. 
For example,   $H^2(I)$ denotes the Sobolev space $W^{2,2}(I)$ on the 
interval $I=(0,1)$, and $H^0(I)=L^2(I)$.  Let $(\cdot, \cdot)_I$ denote 
the $L^2$-inner product on $I$. 
The triple $(\Omega, \mathcal{F}, {\mathbb P})$
stands for a given probability space. For a random variable $X$,  
we denote by $\mathbb{E}[X]$ the expected value of $X$.

We first quote the following existence and uniqueness result from 
\cite{Sarhir_Renesse} for the SPDE
 \eqref{e1.3} with periodic boundary conditions.

\begin{theorem} \label{thm2.1}
Suppose that $u_0\in H^1(I)$ and fix $T > 0$.  
Let $\epsilon \leq \sqrt{2}$. There exists a unique
strong solution to SPDE  \eqref{e1.3a} with periodic boundary conditions 
and attaining the initial condition $u(0)= u_0$,
that is, there exists a unique $H^1$-valued 
$\{ {\mathcal F}_t\}_{t \in [0,T]}$-adapted 
process $u \equiv \{ u(t);\, t \in [0,T]\}$
such that  ${\mathbb P}$-almost surely 
\begin{eqnarray}\label{e2.1}
\quad \bigl(u(t), \varphi \bigr)_I &=&\bigl( u_0,\varphi \bigr)_I
-\frac{\epsilon^2}2 \int_0^t  \bigl( \p_x u, \p_x \varphi \bigr)_I\, ds \\
&&
-\bigl( 1-\frac{\epsilon^2}2\bigr) \int_0^t \bigl( \arctan(\p_x u), 
\p_x \varphi \bigr)_I  \Bigr] ds \nonumber  \\ \nonumber
 &&
 + \epsilon \int_0^t  \Bigl( \sqrt{1+|\p_x u|^2}, \varphi \Bigr)_I   dW_s
 \qquad \forall \varphi\in H^1(I) \quad \forall t\in [0,T].
\end{eqnarray}
Moreover,  $u$ satisfies for some $C > 0$ independent of $T > 0$,
\begin{equation}\label{e2.2}
\sup_{t \in [0,T]} \mathbb{E}  \bigl[\|u(t)\|_{H^1(I)}^2\bigr] \leq C.
\end{equation}

\end{theorem}

It is not clear if such a regularity can be improved from the analysis
of  \cite{Sarhir_Renesse} because of the difficulty caused by the gradient-type 
noise. In particular,  $H^2$-regularity in space, which would be desirable 
in order to derive some rates of convergence for finite element methods, 
seems not clear. To overcome this difficulty, we introduce the following 
simple regularization of \eqref{e1.3}: 
\begin{equation} \label{e2.4}
du^\delta =\Bigl[  \bigl(\delta+\frac{\epsilon^2}2 \bigr) \p_x^2 u^\delta 
+\bigl( 1-\frac{\epsilon^2}2\bigr)
\frac{\p_x^2 u^\delta}{\sqrt{1+|\p_x u^\delta|^2} }  \Bigr] dt  
+ \epsilon \sqrt{1+|\p_x u^\delta|^2}\,  dW_t.
\end{equation}
To make this indirect approach successful, we need to address the 
well-posedness and regularity issues for \eqref{e2.4} and to estimate
the difference between the strong solutions $u^\delta$ of (\ref{e2.4}) 
and $u$ of (\ref{e1.3}).

\begin{theorem} \label{thm2.2}
Suppose that $u^\delta_0\in H^1(I)$ and 
$\|u^\delta_0\|_{H^1(I)}\leq C_0$, where $C_0 > 0$ is independent 
of $\delta$. Let $\epsilon \leq \sqrt{2(1+\delta)}$. Then there exists a unique 
strong solution to SPDE  \eqref{e2.4} with periodic boundary conditions 
and initial condition $u^\delta(0)= u^\delta_0$, that is, there exists a
unique $H^1$-valued $\{ {\mathcal F}_t\}_{t \in [0,T]}$-adapted 
process $u^\delta \equiv \{ u^\delta(t);\, t \in [0,T]\}$
such that  there holds ${\mathbb P}$-almost surely
\begin{align}\label{e2.5}
\quad \bigl(u^\delta(t), \varphi \bigr)_I 
&= \bigl( u^\delta_0,\varphi \bigr)_I 
-\bigl(\delta+\frac{\epsilon^2}2\bigr)
\int_0^t \bigl( \p_x u^\delta, \p_x \varphi \bigr)_I \, ds \\ \nonumber
 &\quad
-\bigl( 1-\frac{\epsilon^2}2\bigr) 
\int_0^t \bigl( \arctan(\p_x u^\delta), \p_x \varphi \bigr)_I  \,ds
  \nonumber  \\
 &\quad
+\epsilon \int_0^t  \Bigl( \sqrt{1+|\p_x u^\delta|^2}, \varphi \Bigr)_I dW_s
 \quad \forall \varphi\in H^1(I) \quad \forall\, t \in [0,T]. \nonumber
\end{align}
Moreover,  $u^\delta$ satisfies
\begin{equation}\label{e2.6}
\sup_{t \in [0,T]} \mathbb{E} \Bigl[\frac12\| \p_x u^\delta(t)\|_{L^2(I)}^2 \Bigr] 
+\delta\, \mathbb{E} \Bigl[\int_0^T  \|\p_x^2 u^\delta(s)\|_{L^2(I)}^2\, ds\Bigr] 
\leq \mathbb{E} \Bigl[\frac12 \| \p_x u^\delta_0\|_{L^2(I)}^2 \Bigr].
\end{equation}

\end{theorem}

\begin{proof}
Existence of $u^{\delta}$ can be shown in the same way as done in 
Theorem~\ref{thm2.1}  (cf. ~\cite{Sarhir_Renesse}).
To verify \eqref{e2.6}, we proceed formally and apply Ito's formula 
with  $f(\cdot) = \frac{1}{2} \Vert \p_x  \cdot \Vert^2_{L^2(I)}$
 to (a Galerkin approximation of) the solution $u^{\delta}$ to get
 \begin{align*}
  & \frac{1}{2} \Vert \p_x u^{\delta}(t)\Vert^2_{L^2(I)} 
+  \int_0^t \Bigl[ \bigl(\frac{\varepsilon^2}{2} + \delta\bigr)
  \Vert \partial_x^2 u^{\delta}\Vert^2_{L^2(I)} 
+ \bigl(1- \frac{\varepsilon^2}{2}\bigr) \Bigl\Vert
  \frac{\partial_x^2 u^{\delta}}{\sqrt{1+ \vert \partial_x u^{\delta}\vert^2}}\Bigr\Vert_{L^2}^2\Bigr]\, ds  \\
  &\qquad = \frac{1}{2} \Vert \p_x u^{\delta}_0\Vert^2_{L^2(I)} 
+ \frac{\varepsilon^2}{2} \int_{0}^t
 \Bigl \Vert \partial_x \sqrt{1 + \vert \partial_x u^{\delta}\vert^2} \Bigr\Vert^2_{L^2} \, ds + M_t \\
  &\qquad = \frac{1}{2} \Vert \p_x u^{\delta}_0\Vert^2_{L^2(I)} 
+ \frac{\varepsilon^2}{2} \int_{0}^t
  \Bigl\Vert \frac{\partial_x u^{\delta} \cdot \partial_x^2 u^{\delta}}{\sqrt{1 + \vert \partial_x u^{\delta}\vert^2}}
  \Bigr\Vert_{L^2(I)}^2\, ds + M_t\qquad \forall\, t \in [0,T],
  \end{align*}
 where
\[
 M_t:= \epsilon \int_0^t 
\Bigl( \p_x \sqrt{1 + \vert \partial_x u^{\delta}(s)\vert^2}, \p_x u^\delta \Bigr)_I\, dW_s
\]
is a  martingale. Taking expectation  yields
\begin{align*}
{\mathbb E} \Bigl[\frac{1}{2} \Vert \p_x u^{\delta}(t)\Vert^2_{L^2(I)} 
+ \int_0^t \bigl[\delta \Vert \partial_x^2 u^{\delta}\Vert^2_{L^2}
 &+ \bigl(1- \frac{\varepsilon^2}{2} \bigr) 
\Bigl\Vert   \frac{\partial_x^2 u^{\delta}}{\sqrt{1+ \vert \partial_x u^{\delta}\vert^2}}\Bigr\Vert_{L^2}^2\bigr] \, ds \Bigr]  \\
& \leq {\mathbb E}\bigl[ \frac{1}{2} \Vert \p_x u^{\delta}_0\Vert^2_{L^2}\bigr].
 \end{align*}
 Hence, \eqref{e2.6} hold. The proof is complete.
 \end{proof}

Next, we shall derive an upper bound for the error $u^\delta-u$ 
as  a low order power function of $\delta$.

\begin{theorem} \label{thm2.3}
Suppose that $u^\delta_0\equiv u_0$. Let $u$ and $u^\delta$ denote 
respectively  the strong solutions of the initial-boundary
value problems \eqref{e1.3} and \eqref{e2.4} as stated in 
Theorems \ref{thm2.1} and \ref{thm2.2}.  Then there holds
the following error estimate:
 \begin{align} \label{e2.8}
\sup_{t\in [0,T]} \mathbb{E}  \Bigl[\|u^\delta(t)-u(t)\|_{L^2(I)}^2  \Bigr]
+ \delta \, \mathbb{E}\Bigl[ \int_0^T \|\p_x \bigl(u^\delta(s)-u(s) \bigr) 
\|_{L^2(I)}^2\, ds \Bigr]   \leq CT \delta.
 \end{align}
\end{theorem}

\begin{proof}
Let $e^\delta:= u^\delta - u$.  Subtracting \eqref{e2.1} 
from \eqref{e2.5} we get that ${\mathbb P}$-a.s.
\begin{align*}
\bigl(e^\delta(t), \varphi \bigr)_I 
&=  - \int_0^t   \Bigl[   \delta  \bigl(  \p_x u, \p_x \varphi \bigr)_I
+\bigl(\delta+\frac{\epsilon^2}2 \bigr)  
\bigl( \p_x e^\delta, \p_x \varphi \bigr)_I  \\
&\hskip 0.3in
+\bigl( 1-\frac{\epsilon^2}2\bigr) \bigl( \arctan(\p_x u^\delta) 
-\arctan(\p_x u), \p_x \varphi \bigr)_I  \Bigr] ds  + M_t
 \end{align*}
 for all $\varphi\in H^1(I)$ and $t \in [0,T]$, with the martingale
 \[
 M_t:= \epsilon \int_0^t  \bigl( \sqrt{1+|\p_x u^\delta|^2}-\sqrt{1+|\p_x u|^2}, \varphi \bigr)_I   dW_s.
 \]

By It\^o's formula we get
\begin{align}\label{e2.9}
\|e^\delta(t)\|_{L^2(I)}^2 &=  - 2 \int_0^t   \Bigl[   \delta  \bigl(  \p_x u, \p_x e^\delta \bigr)_I
+\bigl(\delta+\frac{\epsilon^2}2 \bigr)  \Vert \p_x e^\delta  \bigr \Vert^2_{L^2(I)}  \\ \nonumber
 &\qquad
 +\bigl( 1-\frac{\epsilon^2}2\bigr) \bigl( \arctan(\p_x u^\delta) -\arctan(\p_x u), \p_x e^\delta \bigr)_I  \Bigr] ds  \\ \nonumber
 &\qquad
 + \epsilon^2 \int_0^t  \Bigl\| \sqrt{1+|\p_x u^\delta|^2}-\sqrt{1+|\p_x u|^2} \Bigr\|_{L^2(I)}^2\, ds   \\
 &\qquad
+ 2\epsilon \int_0^t  \Bigl( \sqrt{1+|\p_x u^\delta|^2}-\sqrt{1+|\p_x u|^2}, e^\delta  \Bigr)_I   dW_s.    \nonumber
 \end{align}
Taking expectations on both sides, and using the monotonicity property 
of the $\arctan$ function and the inequality 
$\bigl( \sqrt{1+x^2} - \sqrt{1+y^2}\bigr)^2 \leq \vert x - y\vert^2$ yield 
\begin{align*}
\mathbb{E} \bigl[ \|e^\delta(t)\|_{L^2(I)}^2 \bigr]  &+ 2 \delta \, \mathbb{E} \Bigl[\int_0^t  \|\p_x e^\delta\|_{L^2(I)}^2 \, ds\Bigr]
 \leq  - 2 \delta\, \mathbb{E} \Bigl[\int_0^t   \bigl(  \p_x u, \p_x e^\delta \bigr)_I ds\Bigr]  \\
 \hskip 1.1in
 &\leq \delta\,  \mathbb{E} \Bigl[\int_0^T \Bigl[ \| \p_x u\|_{L^2(I)}^2 + \|\p_x e^\delta \|_{L^2(I)}^2 \Bigr]\, ds\Bigr],
 \end{align*}
 which and \eqref{e2.2} imply that
 \begin{align*}
\mathbb{E}  \bigl[\|e^\delta(t)\|_{L^2(I)}^2\bigr]  +  \delta \, \mathbb{E} \Bigl[\int_0^t  \|\p_x e^\delta\|_{L^2(I)}^2 \, ds\Bigr]
 &\leq \delta\, \mathbb{E} \Bigl[\int_0^T  \| \p_x u\|_{L^2(I)}^2  \, ds\Bigr] \\
 &\leq (CT) \delta.
 \end{align*}
 The desired estimate \eqref{e2.8} follows immediately. The proof is complete.
 \end{proof}

\section{Finite element methods} \label{sec-3}
In this section we propose a fully discrete finite element method 
to solve the regularized SPDE \eqref{e2.4} and  to derive an error 
estimate for the finite element solution.  This goal will be achieved
in two steps. We first present and study a semi-discrete in space 
finite element method and then discretize it in time to obtain our 
fully discrete finite element method.

\subsection{Semi-discretization in space}\label{sec-3.1}
Let $0=x_0<x_1<\cdots<x_{J+1}=1$ be a quasiuniform partition of $I=(0,1)$.  
Define the finite element spaces 
\[
V^h_r:=\bigl\{v_h\in C^0(\overline{I});\, v_h|_{[x_j,x_{j+1}]} \in P_r([x_j, x_{j+1}]), \,  j=0,1, \cdots, J  \bigr\},
\]
where $P_r([x_j, x_{j+1}]$ denotes the space of all polynomials of 
degree not exceeding $r (\geq 0$) on $[x_j,x_{j+1}]$. Our semi-discrete 
finite element method for SPDE \eqref{e2.4} is defined by seeking
$u_h(\cdot, t, \omega): [0,T]\times \Omega \to V^h_r$  
such that ${\mathbb P}$-almost surely \begin{align}\label{eq3.1}
\bigl(u_h^\delta(t), v_h \bigr)_I &= \bigl( u_h^\delta(0), v_h \bigr)_I
- \bigl(\delta+\frac{\epsilon^2}2\bigr) 
\int_0^t   \bigl( \p_x u_h^\delta, \p_x v_h \bigr)_I \, ds \\
 &\qquad
 -\bigl( 1-\frac{\epsilon^2}2\bigr) \int_0^t  
\bigl( \arctan(\p_x u_h^\delta), \p_x v_h \bigr)_I  \, ds \nonumber \\
 &\qquad
 + \epsilon \int_0^t  \Bigl( \sqrt{1+|\p_x u_h^\delta|^2}, v_h \Bigr)_I  dW_s
  \qquad \forall v_h\in V^h_r
 \quad \forall t \in [0,T],   \nonumber \\
 u_h^\delta(t, 0) &=u_h^\delta(t, 1) \qquad \forall t\in [0,T], \label{eq3.1x}
\end{align}
where $u_h^\delta(0)=P_h^r u_0^\delta$, and $P_h^r$ denotes the 
$L^2$-projection operator from $L^2(I)$ to $V_r^h$.

To rewrite the above weak form in the equation form, we introduce 
the discrete (nonlinear)  operator
$A_h^\delta: V_r^h\to V_r^h$ by
\begin{align}\label{eq3.1a}
\bigl( A_h^\delta w_h, v_h \bigr)_I  &:=  \bigl(\delta+\frac{\epsilon^2}2\bigr)
 \bigl( \p_x w_h, \p_x v_h \bigr)_I  \\
 &\hskip 0.3in
 +\bigl( 1-\frac{\epsilon^2}2\bigr) \bigl( \arctan(\p_x w_h), \p_x v_h \bigr)_I
 \quad\forall w_h, v_h\in V_r^h.  \nonumber
 \end{align}
Then \eqref{eq3.1} can be equivalently written as
 \begin{align}\label{eq3.1b}
 d u_h^\delta(t) = - A_h^\delta u_h(t) \, dt 
+ \epsilon P_h\Bigl( \sqrt{1+|\p_x u_h^\delta(t)|^2} \Bigr)\,  dW_t.
 \end{align}

\begin{proposition}\label{prop3.1}
For $\epsilon \le \sqrt{2(1+\delta)}$, there is a unique solution 
$u_h^{\delta} \in C \bigl( [0,T]; L^2(\Omega;V_r^h)\bigr)$ to 
scheme \eqref{eq3.1}.  Moreover, there holds
\begin{align}\label{eq3.2}
&\sup_{0\leq t\leq T} \mathbb{E} \Bigl[  \frac{1}{2} 
\bigl\| u_h^\delta(t) \bigr\|_{L^2(I)}^2 \Bigr]
+ \delta\, \mathbb{E} \Bigl[\int_0^T  \bigl\| \p_x u_h^\delta(s) 
\bigr\|_{L^2(I)}^2\, ds \Bigr]  \\
&\hskip 2in
\leq  \mathbb{E} \Bigl[\frac{1}{2} \bigl\| u_h^{\delta}(0) 
\bigr\|_{L^2(I)}^2 \Bigr] + \epsilon^2 T.  \nonumber
\end{align}
\end{proposition}

\begin{proof}
Well-posedness of (\ref{eq3.1b}) follows from the standard theory 
for stochastic ODEs with Lipschitz drift and diffusion.
To verify \eqref{eq3.2},  applying It\^o's formula to  
$f(u_h^\delta)=\|u_h^\delta \|_{L^2(I)}^2$ and using \eqref{eq3.1b} we get
\begin{align}\label{eq3.2c}
\| u_h^\delta (t)\|_{L^2(I)}^2 &= \| u_h^\delta (0)\|_{L^2(I)}^2 
-2 \int_0^t  \Bigl( A_h^\delta u_h^\delta(s), u_h^\delta(s) \Bigr)_I \, ds \\
&\qquad
+\epsilon^2  \int_0^t  \Bigl\| P_h^r \sqrt{1+|\p_x u_h^\delta(s)|^2 } 
\Bigr\|_{L^2(I)}^2 \, ds  \nonumber \\
&\qquad
+2\epsilon \int_0^t  \Bigl( P_h^r\sqrt{1+|\p_x u_h^\delta(s)|^2}, 
u_h^\delta \Bigr)_I\,  dW_s .  \nonumber
\end{align}
It follows from the definitions of $A_h^{\delta}$  and $P_h^r$ that
\begin{align}\label{eq3.2d}
\| u_h^\delta (t)\|_{L^2(I)}^2 &\leq \| u_h^\delta (0)\|_{L^2(I)}^2
-(2\delta+\epsilon^2)\int_0^t   \bigl\| \p_x u_h^\delta(s) \bigr\|_{L^2(I)}^2\, ds \\
&\hskip 0.2in
- (2-\epsilon^2) \int_0^t  \Bigl( \arctan(\p_x u_h^\delta(s)), \p_x u_h^\delta(s)  \Bigr)_I  \, ds \nonumber  \\
&\hskip 0.2in
+ \epsilon^2  \int_0^t  \Bigl[  1+ \bigl\| \p_x u_h^\delta(s) \bigr\|_{L^2(I)}^2  \Bigr] \, ds  \nonumber \\
&\hskip 0.2in
+2\epsilon \int_0^t  \Bigl( \sqrt{1+|\p_x u_h^\delta(s)|^2}, u_h^\delta \Bigr)_I\,  dW_s.  \nonumber
\end{align}
Then \eqref{eq3.2} follows from applying expectation to \eqref{eq3.2d}, and
using the coercivity of arctan.  The proof is complete.
\end{proof}

An a priori estimate for $u_h^\delta$ in stronger norms is more difficult 
to obtain, which is due to low global smoothness and local nature 
of finite element functions.  We shall derive some of these estimates
in Proposition \ref{prop3.7} using a perturbation argument after 
establishing error estimates for $u_h^\delta$.

To derive error estimates for $u_h^\delta$, we introduce  the 
elliptic $H^1$-projection $R_h^r: H^1(I)\to V_r^h$, i.e., for 
any $w \in H^1(I)$, $R_h^r w \in V_r^h $ is defined by 
\begin{equation}\label{eq3.3}
\bigl( \p_x [R_h^r w - w], \p_x v_h \bigr)_I
+ \bigl(R_h^r w -w, v_h \bigr)_I = 0 \qquad \forall v_h\in V_r^h.
\end{equation}
The following error bounds are well-known (cf.~\cite{Brenner_Scott02}),
\begin{align}\label{eq3.4}
\bigl\| w-   R_h^r w \bigr\|_{L^2(I)} 
+ h \bigl\| w -R_h^r w \bigr\|_{H^1(I)}  \leq C h^2 \|w\|_{H^2(I)}.
\end{align}

\begin{theorem}\label{thm3.1}
Let  $\epsilon \leq \sqrt{2(1+\delta)}$. Then there holds
\begin{align}\label{eq3.5a}
\sup_{t\in [0,T]}  \mathbb{E}  \Bigl[\bigl\| u^\delta(t) -  u_h^\delta(t) \bigr\|_{L^2(I)}^2\Bigr]
&+\delta \, \mathbb{E} \Bigl[\int_0^T \bigl\|\p_x [u^\delta(s) -  u_h^\delta(s)] \bigr\|_{L^2(I)}^2 \, ds \Bigr] \\
& \leq C h^2  \bigl( 1+ \delta^{-2} \bigr).  \nonumber
\end{align}
\end{theorem}

\begin{proof}
Let
\[
e^{\delta}(t):=u^\delta(t)-u_h^{\delta}(t) ,\quad
\eta^{\delta}: = u^\delta(t)-R_h^r u^{\delta}(t),\quad
\xi^{\delta}:= R_h^r u^{\delta}(t)-u_h^{\delta}(t).
\]
Then $e^{\delta}=\eta^{\delta} + \xi^{\delta}$.  Subtracting \eqref{eq3.1} from  \eqref{e2.5} we obtain
the following error equation which holds ${\mathbb P}$-almost surely:
\begin{align}\label{eq3.6}
&\bigl( e^\delta(t), v_h \bigr)_I +  \bigl(\delta+\frac{\epsilon^2}2\bigr)
\int_0^t  \bigl( \p_x e^\delta(s), \p_x v_h \bigr)_I  ds \\
 &\hskip 0.0in
 = - \bigl( 1-\frac{\epsilon^2}2\bigr) \int_{0}^{t}  \Bigl( \arctan(\p_x u^\delta(s)) -\arctan(\p_x u_h^\delta(s)), \p_x v_h \Bigr)_I  ds
 \nonumber \\
 &\hskip 0.1in
 + \epsilon \int_{0}^{t}   \Bigl( \sqrt{1+|\p_x u^\delta(s)|^2}-\sqrt{1+|\p_x u_h^\delta(s)|^2},  v_h\Bigr)_I   dW_s
 + \bigl( e^\delta(0), v_h \bigr)_I  \nonumber
\end{align}
for all $v_h\in V_r^h$.  Substituting $e^\delta=\eta^\delta+\xi^\delta$ and rearranging terms leads to
\begin{align}\label{eq3.6b}
&\bigl( \xi^\delta(t), v_h \bigr)_I +  \bigl(\delta+\frac{\epsilon^2}2\bigr)
\int_0^t  \bigl( \p_x \xi^\delta(s), \p_x v_h \bigr)_I  ds \\
&\hskip 0.2in + \bigl( 1-\frac{\epsilon^2}2\bigr) \int_{0}^{t}  \Bigl( \arctan \bigl(\p_x u^\delta(s)\bigr)
-\arctan \bigl(\p_x u_h^\delta(s)\bigr), \p_x v_h \Bigr)_I  ds  \nonumber \\
 &\hskip 0.0in
 = \epsilon \int_{0}^{t}   \Bigl( \sqrt{1+|\p_x u^\delta(s)|^2}-\sqrt{1+|\p_x u_h^\delta(s)|^2},  v_h\Bigr)_I   dW_s  \nonumber \\
 &\hskip 0.2in 
-\bigl(\delta +\frac{\epsilon^2}{2} \bigr) \int_0^t \bigl(\eta^\delta(s),v_h\bigr)_I\, ds -\bigl( \eta^{\delta}(t), v_h \bigr)_I 
+ \bigl( e^\delta(0), v_h \bigr)_I.  \nonumber
\end{align}

Applying It\^o's formula with $f(\xi^\delta)=\|\xi^\delta\|_{L^2(I)}^2$,
and using \eqref{eq3.6b} and \eqref{eq3.1a} we obtain
\begin{align}\label{eq3.7}
& \bigl\| \xi^{\delta}(t) \bigr\|_{L^2(I)}^2  +  \bigl(2\delta+\epsilon^2\bigr)
\int_0^t  \bigl\| \p_x \xi^\delta(s) \bigr\|_{L^2(I)}^2\,  ds \\  \nonumber
&\hskip 0.4in
+ (2-\epsilon^2) \int_0^t \Bigl( {\rm arctan} \bigl(\partial_x R_h^r u^{\delta}(s) \bigr) - {\rm arctan} \bigl(\partial_x  u_h^{\delta}(s)\bigr),
\partial_x \xi^{\delta}(s) \Bigr)_I ds  \nonumber \\
 &\hskip 0.2in
= -\bigl( 2-\epsilon^2\bigr) \int_0^t \Bigl( \arctan \bigl(\p_x  u^\delta(s) \bigr)-\arctan \bigl(\p_x R^r_h u^{\delta}(s)\bigr), \p_x \xi^{\delta}(s)\Bigr)_I  ds  \nonumber \\
&\hskip 0.4in
+ \epsilon^2 \int_0^t \Bigl\| \sqrt{1+|\p_x u^\delta(s)|^2}-\sqrt{1+|\p_x u_h^\delta(s)|^2} \Bigr\|_{L^2(I)}^2\, ds \nonumber\\
&\hskip 0.4in
+2 \epsilon \int_0^t   \Bigl( \sqrt{1+|\p_x u^\delta(s)|^2}-\sqrt{1+|\p_x u_h^\delta(s)|^2}, \xi^\delta(s) \Bigr)_I   dW_s  \nonumber  \\
&\hskip 0.4in
-\bigl(2\delta +\epsilon^2\bigr) \int_0^t \bigl(\eta^\delta(s),\xi_h\bigr)_I\, ds 
-2\bigl( \eta^{\delta}(t), \xi^{\delta}(t) \bigr)_I  + 2\bigl( e^\delta(0), \xi^\delta(t)  \bigr)_I.  \nonumber
 \end{align}

By the monotonicity of arctan,  \eqref{eq3.4}, (\ref{e2.6}), and the inequality $\bigl( \sqrt{1+x^2} - \sqrt{1+y^2}\bigr)^2 \leq \vert x - y\vert^2$ we have
\begin{align*}
&{\mathbb E} \Bigl[\int_0^t \Bigl( \arctan(\p_x R_h^r u^\delta(s))-\arctan(\p_x u_h^{\delta}(s)), \p_x \xi^{\delta}(s)\Bigr)_I  ds \Bigr] \geq 0,\\
&(2-\epsilon^2) {\mathbb E}\Bigl[\int_0^t \Bigl( \arctan \bigl(\p_x  u^\delta(s) \bigr)-\arctan \bigl(\p_x R^r_h u^{\delta}(s)\bigr), \p_x \xi^{\delta}(s)\Bigr)_I  ds\Bigr] \\
&\hskip 0.5in
\leq {\mathbb E}\Bigl[\int_0^t  \Bigl( \frac{\delta}4 \| \p_x \xi^\delta(s)\|_{L^2(I)}^2
+  4\delta^{-1}  \bigl\| \p_x u^\delta(s)-\p_x R_h^r u^\delta(s) \bigr\|_{L^2(I)}^2 \Bigr)\, ds\Bigr] \\
&\hskip 0.5in
\leq  \frac{\delta}4 {\mathbb E}\Bigl[\int_0^t   \| \p_x \xi^\delta(s)\|_{L^2(I)}^2\, ds\Bigr] + Ch^2 \delta^{-2}, \\
&{\mathbb E}\Bigl[\epsilon^2 \int_0^t \Bigl\| \sqrt{1+|\p_x u^\delta(s)|^2}-\sqrt{1+|\p_x u_h^\delta(s)|^2} \Bigr\|_{L^2(I)}^2\, ds\Bigr] \\
&\hskip 0.5in
\leq  {\mathbb E}\Bigl[\bigl(\epsilon^2 + \frac{\delta}{4} \bigr)  \int_0^t   \| \p_x \xi^\delta(s)\|_{L^2(I)}^2\, ds
\Bigr]
+ C \delta^{-1}  {\mathbb E}\Bigl[\int_0^t  \| \p_x \eta^\delta \|_{L^2(I)}^2\, ds\Bigr]  \\
&\hskip 0.5in
\leq  \bigl(\epsilon^2 + \frac{\delta}{4} \bigr)  {\mathbb E}\Bigl[\int_0^t   \| \p_x \xi^\delta(s)\|_{L^2(I)}^2\, ds
\Bigr] + Ch^2 \delta^{-2}, \\
&{\mathbb E}\Bigl[\bigl( \eta^\delta(t), \xi^\delta(t)  \bigr)_I \Bigr]
\leq  {\mathbb E}\Bigl[\frac14 \|\xi^\delta(t) \|_{L^2(I)}^2 + \| \eta^\delta(t) \|_{L^2(I)}^2 \Bigr]\leq  \frac14
{\mathbb E}\Bigl[\|\xi^\delta(t) \|_{L^2(I)}^2\Bigr] + Ch^2, \\
&{\mathbb E}\Bigl[\bigl( e^\delta(0), \xi^\delta(t)  \bigr)_I \Bigr]
\leq  \frac14 {\mathbb E}\Bigl[\|\xi^\delta(t) \|_{L^2(I)}^2\Bigr] +
{\mathbb E}\Bigl[\| e^\delta(0) \|_{L^2(I)}^2\Bigr] \leq  \frac14 \|\xi^\delta(t) \|_{L^2(I)}^2 + Ch^2.
\end{align*}
Taking the expectation in \eqref{eq3.7} and using the above estimates then  yields
\begin{align}\label{eq3.9}
\sup_{t\in[0,T]}  \mathbb{E} \Bigl[\bigl\| \xi^{\delta}(t) \bigr\|_{L^2(I)}^2  \Bigr]
&+  3\delta \, \mathbb{E} \Bigl[\int_0^T  \bigl\| \p_x \xi^\delta(s) \bigr\|_{L^2(I)}^2\,  ds \Bigr]
\leq  Ch^2 \bigl(  1+\delta^{-2} \bigr).
\end{align}
Finally, \eqref{eq3.5a} follows from the triangle inequality, \eqref{eq3.4}, 
and \eqref{eq3.9}. The proof is complete.
\end{proof}

\begin{remark}
(a)  Estimate \eqref{eq3.5a} is optimal in the $H^1$-norm, but suboptimal 
in the $L^2$-norm.   The suboptimal rate for the $L^2$-error is caused 
by the stochastic effect, i.e., the second term on the right-hand side 
of \eqref{eq3.7}, and it is also caused by the lack of the space-time 
regularity in $L^\infty((0,T); H^2(I))$ for $u^\delta$.

(b) The proof still holds if the elliptic projection $R_h^r$  is replaced 
by the $L^2$-projection $P_h^r$.
\end{remark}

We now use estimate (\ref{eq3.9}) to  derive some stronger norm estimates 
for $u_h^\delta$.  To this end, we define the {\em discrete Laplacian} 
$\partial^2_h: V_r^h  \rightarrow V_r^h $ by
\begin{equation}\label{prop3.6b}
(\partial^2_h w_h, v_h)_I = - (\partial_x w_h, \partial_x v_h)_I 
\qquad \forall w_h,v_h \in V_r^h,
\end{equation}
and the $L^2$-projection $P_h^r: L^2(I) \rightarrow V_r^h $ by
\[
(P_h^r w, v_h)_I = (w,v_h)_I  \qquad \forall v_h \in V^h_r.
\]

\begin{proposition}\label{prop3.7}
For $\varepsilon \leq \sqrt{2(1+\delta)}$ there hold the following estimates 
for the solution $u_h^{\delta}$ of scheme \eqref{eq3.1}:
\begin{align}\label{prop3.7a}
&\sup_{0 \leq t \leq T} {\mathbb E} \Bigl[ \Vert \partial_x u^{\delta}_h(t)\Vert^2_{L^2(I)}\Bigr] +
\delta\, {\mathbb E} \Bigl[ \int_0^T \Vert \partial^2_h u_h^{\delta}(s)\Vert^2_{L^2(I)} \, ds\Bigr]
\leq C\bigl(1+ \delta^{-2}\bigr), \\ \label{prop3.7b}
& {\mathbb E} \Bigl[ \Vert u^\delta_h(t) - u^\delta_h(s)\Vert_{L^2(I)}^2 + \frac{\delta}{2}
\int_s^t \Vert \partial_x[u^\delta_h(\zeta) - u^\delta_h(s)]\Vert^2_{L^2(I)} \, d\zeta\Bigr]  \\ \nonumber
& \hskip 1.4in
 \leq C \bigl(1+ \delta^{-3}\bigr) \vert t-s\vert \qquad \forall\, 0 \leq s \leq t \leq T.
\end{align}
\end{proposition}

\begin{proof}
Notice that $u_h^\delta=\xi^\delta +R_h^r u^\delta$.  By the $H^1$-stability 
of $R_h^r$,  an inverse inequality, (\ref{e2.6}), and (\ref{eq3.9}) we get
\begin{align*}
\sup_{t \in [0,T]} {\mathbb E} \Bigl[ \Vert \partial_x u_h^{\delta} (t) \Vert^2_{L^2(I)}\Bigr]
&\leq 2\sup_{t \in [0,T]} {\mathbb E} \Bigl[ \Vert \partial_x R_h^r u^{\delta} (t) \Vert^2_{L^2(I)}\Bigr] +
2 \sup_{t \in [0,T]} {\mathbb E} \Bigl[ \Vert \p_x  \xi^{\delta} (t) \Vert^2_{L^2(I)}\Bigr] \\
&\leq C\sup_{t \in [0,T]} {\mathbb E} \Bigl[ \Vert \partial_x u^{\delta} (t) \Vert^2_{L^2(I)}\Bigr] +
\frac{C}{h^2} \sup_{t \in [0,T]} {\mathbb E} \Bigl[ \Vert  \xi^{\delta} (t) \Vert^2_{L^2(I)}\Bigr] \\
&\leq C (1+\delta^{-2}).
\end{align*}
It follows from (\ref{prop3.6b}) and (\ref{eq3.3}) that
\begin{align*}
\Vert \partial^2_h R_h^r w\Vert^2_{L^2(I)} 
&= -\bigl( \partial_x \partial^2_h R_h^r w , \partial_x R_h^r w\bigr)_{I}\\
&= \bigl( \partial^2_h R_h^r w, \partial_x^2 w\bigr)_I 
+\bigl(w-R_h^r w, \p_h^2R_h^r w\bigr)_I
\qquad \forall w \in H^2(I), 
\end{align*}
and hence
\begin{equation}\label{3.18}
\Vert \partial^2_h R_h^r w\Vert_{L^2(I)} \leq \Vert \partial_x^2 w\Vert_{L^2(I)}
+\|w-R_h^r w\|_{L^2(I)} \leq (1+Ch^2) \|w\|_{H^2(I)}.  
\end{equation}
By an inverse estimate, (\ref{3.18}), and (\ref{eq3.9})  we have
\begin{align*}
{\mathbb E} \Bigl[ \int_0^T \Vert \partial^2_h u_h^{\delta}(s)\Vert^2_{L^2(I)}\, ds\Bigr]
 &\leq 2{\mathbb E} \Bigl[ \int_0^T \Bigl(\Vert \partial^2_h \xi^{\delta}(s)\Vert^2_{L^2(I)}
 + \Vert \partial^2_h R_h^r u^{\delta}(s)\Vert^2_{L^2(I)}\Bigr) ds\Bigr] \\
 &\leq 2{\mathbb E} \Bigl[ \int_0^T \Bigl(Ch^{-2} \Vert \partial_x \xi^{\delta}(s)\Vert^2_{L^2(I)}
 +C\Vert \partial^2_x u^{\delta}(s)\Vert^2_{L^2(I)}\Bigr) ds\Bigr] \\
 &\leq C \delta^{-1} (1 + \delta^{-2}) 
+C{\mathbb E} \Bigl[\int_0^T\Vert \p_x^2 u^\delta(s) \Vert^2_{L^2(I)}\, ds\Bigr],
\end{align*}
which and  (\ref{e2.6}) give the desired bound in (\ref{prop3.7a}).

To show (\ref{prop3.7b}), we fix $ s\geq 0$ and apply Ito's formula 
to $f(u_h^\delta) = \Vert u_h^\delta(t) - u^{\delta}_h(s)\Vert^2_{L^2(I)}$
to get that for some $\{ {\mathcal F}_t;\, t \in [s,T]\}$-martingale $M_t$
\begin{align}\label{4.9a}
&\Vert u^\delta_h(t) - u^\delta_h(s)\Vert^2_{L^2(I)} \\ \nonumber
&\hskip -0.2in
=- (\epsilon^2 + 2\delta) \int_s^t \Bigl( \partial_x [u_h^{\delta}(\zeta)\pm u^{\delta}_h(s)], \partial_x[u^\delta_h(\zeta) - u^{\delta}_h(s)]\Bigr)_I\, d\zeta \\
& - (2-\epsilon^2) \int_s^t \Bigl( {\rm arctan} \bigl( \partial_x u_h^{\delta}(\zeta)\bigr) \pm {\rm arctan}\bigl( \partial_x u_h^\delta(s)\bigr), \partial_x[u^\delta_h(\zeta) - u^\delta_h(s)]\Bigr)_I\, d\zeta \nonumber \\
&+ \epsilon^2 \int_s^t \bigl\Vert P_h^r \sqrt{1+ \vert \partial_x \bigl[u^\delta_h(\zeta) \pm u^\delta_h(s) \bigr]\vert^2} \bigr\Vert^2_{L^2(I)}\, d\zeta + M_t. \nonumber
\end{align}
By the $L^2$-stability of $P_h^r$, the triangle and Young's inequality,
we can bound the last term above as follows:
\begin{align*}
& \epsilon^2 \int_s^t \bigl\Vert P_h^r \sqrt{1+ \vert \partial_x \bigl[u^\delta_h(\zeta) \pm u^\delta_h(s) \bigr]\vert^2} \bigr\Vert^2_{L^2(I)}\, d\zeta \\
&\hskip 0.5in
 \leq \epsilon^2 \int_s^t \Bigl(\bigl\Vert \partial_x[u^\delta_h(\zeta) - u^\delta_h(s)]\bigr\Vert_{L^2(I)}
+ \bigl\Vert 1+ \vert \partial_x u^\delta_h(s)\vert \bigr\Vert_{L^2(I)}\Bigr)^2\, d\zeta  \nonumber  \\
&\hskip 0.5in
 \leq \epsilon^2(1 +\delta) \int_s^t \Vert \partial_x[u^\delta_h(\zeta) - u^\delta_h(s)]\bigr\Vert_{L^2(I)}^2\, d\zeta \nonumber \\
&\hskip 0.8in
 + \epsilon^2  (4+ \delta^{-1}) \Bigl( 1 + \Vert \partial_x u_h^{\delta}(s)\Vert_{L^2(I)}^2\Bigr) \, |t-s|. \nonumber
\end{align*}
Also
\begin{align*}
&(\epsilon^2 + 2\delta) \int_s^t \Bigl( \partial_x u^\delta_h(s), \partial_x[u_h^\delta(\zeta) - u^\partial_h(s)]\Bigr)_I\, d\zeta \\
&\quad \leq \frac{\delta}4  \int_s^t \Vert \partial_x[u^\delta_h(\zeta) - u^\delta_h(s)]\Vert^2_{L^2(I)}\, d\zeta
+ (\epsilon^2+ 2\delta)^2 \delta^{-1}  \, \vert t-s\vert \Vert \partial_x u_h^\delta(s)\Vert^2_{L^2(I)}, \\
&(2-\epsilon^2) \int_s^t \Bigl( {\rm arctan}\bigl( \partial_x u_h^\delta (s)\bigr), \partial_x[u_h^\delta(\zeta) - u^\partial_h(s)]\Bigr)_I\, d\zeta \\
&\quad \leq \frac{\delta}{4} \int_s^t \Vert \partial_x[u^\delta_h(\zeta) - u^\delta_h(s)]\Vert^2_{L^2(I)}\, d\zeta
+ 4 (2-\epsilon^2)^2 \delta^{-1}\, |t-s|.
\end{align*}
Substituting the above estimates into \eqref{4.9a} yields
\begin{eqnarray*}
&&\Vert u^\delta_h(t)- u_h^\delta(s)\Vert^2_{L^2(I)} + \frac{\delta}{2}
\int_s^t \Vert \partial_x[u^\delta_h(\zeta) - u^\delta_h(s)]\Vert^2\, d\zeta \\
&& \hskip 0.5in
 \leq C \delta^{-1} \Bigl((2-\epsilon^2)^2 + \epsilon^4 + \delta^2 \Bigr)
  \Bigl( 1 + \Vert \partial_x u^\delta_h(s)\Vert_{L^2(I)}^2\Bigr) \,
 | t-s| + M_t.
\end{eqnarray*}
Finally, (\ref{prop3.7b}) follows from applying the expectation to the 
above inequality and using (\ref{prop3.7a}).
\end{proof}

\subsection{Fully discrete finite element methods}\label{sec-3.2}
Let $t_n = n\tau$ for $n=0,1,\cdots,N$ be a uniform partition of $[0,T]$ 
with $\tau=T/N$. Our fully discrete finite element method for 
SPDE \eqref{e2.4} is defined by seeking an $\{\mathcal{F}_{t_n};
n=0,1,\cdots,N\}$-adapted $V_r^h$-valued process 
$\{u_h^n;\,  n=0,1,\cdots, N\}$ 
such that such that $\mathbb{P}$-almost surely
\begin{align}\label{e3.1}
&\bigl(u_h^{\delta, n+1}, v_h\bigr)_I  +\tau\bigl(\delta+\frac{\epsilon^2}2 \bigr) \bigl( \p_x u_h^{\delta,n+1}, \p_x v_h\bigr)_I\\
&\hskip 0.8in
+\tau\bigl( 1-\frac{\epsilon^2}2\bigr)  \bigl( \arctan(\p_x u_h^{\delta, n+1}), \p_x v_h \bigr)_I  \nonumber  \\
&\hskip 0.3in
= \bigl( u_h^{\delta,n}, v_h\bigr)_I  + \epsilon \Bigl( \sqrt{1+|\p_x u_h^{\delta,n}|^2}, v_h \Bigr)_I\, \Delta  W_{n+1}
\qquad\forall v_h\in V^h_r,  \nonumber \\
\hskip 0.5in
&u_h^{\delta,n+1}(0) = u_h^{\delta, n+1}(1),  \label{e3.1a}
\end{align}
where $\Delta W_{n+1}:= W(t_{n+1})-W(t_n) \sim {\mathcal N}(0,\tau)$.

We first establish the following stability estimate for $u_h^{\delta,n}$.

\begin{proposition}\label{prop3.3}
Let $\epsilon \leq \sqrt{2(1+\delta)}$.
For each $n=0,1,\cdots, N$, there is a $V_r^h$-valued discrete 
process $\{u_h^{\delta,n+1};\,  0 \leq n\leq N-1\}$
which solves scheme \eqref{e3.1}--\eqref{e3.1a}.  Moreover, there holds
\begin{align}\label{e3.2}
\max_{0\leq n\leq N} \mathbb{E} \Bigl[\bigl\| u_h^{\delta,n} 
\bigr\|_{L^2(I)}^2 \Bigr] + 2\delta \sum_{n=0}^{N} \tau 
\mathbb{E}\Bigl[ \bigl\| \p_x u_h^{\delta,n} \bigr\|_{L^2(I)}^2 \Bigr]
\leq   \mathbb{E}\Bigl[\bigl\| u_h^{\delta,0} \bigr\|_{L^2(I)}^2\Bigr] +\epsilon^2 T.
\end{align}
\end{proposition}

\begin{proof}
The existence of solutions to scheme \eqref{e3.1}--\eqref{e3.1a} 
for  $\tau,h > 0$ can be proved by Brouwer's fixed-point theorem, 
which uses the coercivity of the operator $I + \tau A_h^\delta$ 
(see ~\eqref{eq3.1a}).

To show \eqref{e3.2}, we choose $v_h=u_h^{\delta,n+1}(\omega)$ 
in \eqref{e3.1} to find $\mathbb{P}$-almost surely
\begin{align}\label{e3.3}
&\frac12 \Bigl[  \bigl\| u_h^{\delta,n+1} \bigr\|_{L^2(I)}^2 -\bigl\| u_h^{\delta,n} \bigr\|_{L^2(I)}^2 \Bigr]
+ \frac12 \bigl\| u_h^{\delta,n+1} - u_h^{\delta,n} \bigr\|_{L^2(I)}^2 \\
&\hskip 0.2in
+ \tau\bigl( \delta+\frac{\epsilon^2}2\bigr) \bigl\| \p_x u_h^{\delta,n+1} \bigr\|_{L^2(I)}^2
+ \tau\bigl( 1-\frac{\epsilon^2}2\bigr)  \bigl( \arctan(\p_x u_h^{\delta, n+1}), \p_x u_h^{\delta, n+1} \bigr)_I  \nonumber \\
&\hskip 0.1in
=\epsilon \Bigl( \sqrt{1+|\p_x u_h^{\delta,n}|^2}, u_h^{\delta,n} + u_h^{\delta,n+1}-u_h^{\delta,n} \Bigr)_I\, \Delta  W_{n+1}. \nonumber
\end{align}
We compute
\begin{align*}
& \bigl( \arctan(\p_x u_h^{\delta, n+1}), \p_x u_h^{\delta, n+1} \bigr)_I \geq 0,\\
&\epsilon \Bigl( \sqrt{1+|\p_x u_h^{\delta,n}|^2},  u_h^{\delta,n+1}-u_h^{\delta,n} \Bigr)_I\, \Delta  W_{n+1} \\
&\hskip 0.5in
\leq \frac12 \bigl\| u_h^{\delta,n+1} - u_h^{\delta,n} \bigr\|_{L^2(I)}^2
+ \frac{\epsilon^2}{2} \Bigl\|  \sqrt{1+ |\p_x u_h^{\delta,n}|^2} \Bigr\|_{L^2(I)}^2 |\Delta W_{n+1}|^2.
\end{align*}
By the tower property for expectations, there holds
\begin{equation*}
\frac{\epsilon^2}{2} {\mathbb E} \Bigl[ \Bigl\| 
\sqrt{1+ \vert \partial_x u_h^{\delta,n}\vert^2} \Bigr\|_{L^2(I)}^2
{\mathbb E}\bigl[\vert \Delta W_{n+1}\vert^2\vert{\mathcal F}_{t_n}\bigr] \Bigr]
= \frac{\epsilon^2}{2} \tau \, {\mathbb E} \Bigl[ 1 + \Vert \partial_x u^{\delta,n}\Vert^2_{L^2(I)}\Bigr],
\end{equation*}
such that we get
\begin{eqnarray}\label{e3.4}
&&\frac12 \mathbb{E}\, \Bigl[  \bigl\| u_h^{\delta,n+1} \bigr\|_{L^2(I)}^2 -\bigl\| u_h^{\delta,n} \bigr\|_{L^2(I)}^2 \Bigr]
+ \tau \delta \, \mathbb{E} \Bigl[\bigl\| \p_x u_h^{\delta,n+1} \bigr\|_{L^2(I)}^2 \Bigr] \\ \nonumber
&&\qquad + \frac{\epsilon^2}{2} \tau\, {\mathbb E} \Bigl[ \bigl\| \partial_x u_h^{\delta,n+1} \bigr\|_{L^2(I)}^2 -\bigl\| \partial_x u_h^{\delta,n} \bigr\|_{L^2(I)}^2\Bigr]
 \leq \epsilon^2 \tau.
\end{eqnarray}
After summation, we arrive at
\begin{align*}
\max_{0\leq n\leq N} \mathbb{E} \bigl[  \bigl\| u_h^{\delta,n} 
\bigr\|_{L^2(I)}^2 \bigr] + 2\delta \tau \sum_{n=0}^{N}  \mathbb{E} 
\Bigl[ \bigl\| \p_x u_h^{\delta,n} \bigr\|_{L^2(I)}^2 \Bigr]
\leq \mathbb{E}\Bigl[\bigl\| u_h^{\delta,0} \bigr\|_{L^2(I)}^2\Bigr] +\epsilon^2 T.
\end{align*}
So \eqref{e3.2} holds.  The proof is complete.
\end{proof}

Next, we derive an error bound for $u_h^\delta(t_n)-u_h^{\delta,n}$.

\begin{theorem}\label{thm3.2}
There holds the following error estimate:
\begin{align}\label{e3.5}
\sup_{0\leq n\leq N}  \mathbb{E}  \Bigl[\bigl\| u_h^\delta(t_n) -  u_h^{\delta,n} \bigr\|_{L^2(I)}^2 \Bigr]
&+ \delta \, \mathbb{E} \Bigl[\sum_{n=0}^N \tau \bigl\|\p_x u_h^\delta(t_n) - \p_x u_h^{\delta,n} \bigr\|_{L^2(I)}^2 \Bigr] \\
&\leq CT\bigl(1+ \delta^{-2} \bigr) h^{-2} \tau.  \nonumber
\end{align}
\end{theorem}

\begin{proof}
Let $e^{\delta,n}:=u_h^\delta(t_n)-u_h^{\delta,n}$.  It follows 
from \eqref{eq3.1} that for all $\{ t_n; n \geq 0\}$ there 
holds ${\mathbb P}$-almost surely
\begin{align}\label{e3.6}
&\bigl( u_h^\delta(t_{n+1}), v_h \bigr)_I - \bigl( u_h^\delta(t_n), v_h \bigr)_I  \\
&\hskip 0.9in
= -\bigl(\delta+\frac{\epsilon^2}2\bigr)  \int_{t_n}^{t_{n+1}}  \bigl( \p_x u_h^\delta(s), \p_x v_h \bigr)_I ds \nonumber   \\
 &\hskip 1.2in
 -\bigl( 1-\frac{\epsilon^2}2\bigr) \int_{t_n}^{t_{n+1}} \bigl( \arctan(\p_x u_h^\delta(s)), \p_x v_h \bigr)_I  ds  \nonumber \\
 &\hskip 1.2in
 + \epsilon \int_{t_n}^{t_{n+1}}   \Bigl( \sqrt{1+|\p_x u_h^\delta(s)|^2},  v_h\Bigr)_I   dW_s
 \qquad \forall v_h\in V_r^h .   \nonumber
\end{align}

Subtracting \eqref{e3.1} from \eqref{e3.6}  yields the following 
error equation:
\begin{align}\label{e3.7}
&\bigl( e^{\delta,n+1}, v_h \bigr)_I - \bigl( e^{\delta, n}, v_h \bigr)_I  \\
&\hskip 0.2in
= -\bigl(\delta+\frac{\epsilon^2}2\bigr)  \int_{t_n}^{t_{n+1}}  \bigl( \p_x u_h^\delta(s)-\p_x u_h^{\delta,n+1}, \p_x v_h \bigr)_I ds
 \nonumber \\
 &\hskip 0.4in
 -\bigl( 1-\frac{\epsilon^2}2\bigr)  \int_{t_n}^{t_{n+1}} \Bigl( \arctan \bigl(\p_x u_h^\delta(s) \bigr)-\arctan(\p_x u_h^{\delta,n+1}), \p_x v_h \Bigr)_I  ds
 \nonumber \\
 &\hskip 0.4in
 + \epsilon \int_{t_n}^{t_{n+1}}   \Bigl( \sqrt{1+|\p_x u_h^\delta(s)|^2} - \sqrt{1+|\p_x u_h^{\delta, n}|^2},  v_h\Bigr)_I   dW_s . \nonumber
\end{align}
Choosing $v_h=e^{\delta,n+1}(\omega)$ in \eqref{e3.7} 
leads to $\mathbb{P}$-almost surely
\begin{align}\label{e3.8}
&\frac12\bigl[ \| e^{\delta,n+1} \|_{L^2(I)}^2 - \| e^{\delta,n} \|_{L^2(I)}^2 \bigr]
+\frac12 \bigl\| e^{\delta,n+1}  -  e^{\delta,n} \bigr\|_{L^2(I)}^2 \\
&\hskip 1in
+ \bigl(\frac{\epsilon^2}2 + \delta\bigr) \tau\, \Vert \p_x e^{\delta,n+1}\Vert^2_{L^2(I)} \nonumber \\
&\hskip 0.1in
= -  \bigl(\frac{\epsilon^2}2+\delta\bigr) \int_{t_n}^{t_{n+1}}
 \Bigl( \p_x u_h^\delta(s)- \p_x u_h^\delta(t_{n+1}) , \p_x e^{\delta,n+1} \Bigr)_I ds  \nonumber \\
 &\hskip 0.25in
 -\bigl( 1-\frac{\epsilon^2}2\bigr) \int_{t_n}^{t_{n+1}} \Bigl( \arctan \bigl(\p_x u_h^\delta(s) \bigr)
 \pm {\rm arctan} \bigl( \partial_x u_h^{\delta}(t_{n+1})\bigr) \nonumber \\
 &\hskip 2.4in -\arctan(\p_x u_h^{\delta,n+1}), \p_x e^{\delta, n+1} \Bigr)_I ds  \nonumber \\
 &\hskip 0.25in
 + \epsilon \int_{t_n}^{t_{n+1}}   \Bigl( \sqrt{1+|\p_x u_h^\delta(s)|^2} - \sqrt{1+|\p_x u_h^{\delta, n}|^2},  e^{\delta, n+1}\pm e^{\delta,n} \Bigr)_I   dW_s.\nonumber
\end{align}

We now bound each term on the right-hand side.  First, since 
${\mathbb E}[\Delta W_{n+1}] = 0$, by Ito's isometry, the inequality 
$\bigl(\sqrt{1+x^2} - \sqrt{1+y^2} \bigr)^2 \leq (x-y)^2$, and the inverse 
inequality we get 
\begin{align} \nonumber
&\mathbb{E}\Bigl[ \epsilon \int_{t_n}^{t_{n+1}}   \Bigl( \sqrt{1+|\p_x u_h^\delta(s)|^2} - \sqrt{1+|\p_x u_h^{\delta, n}|^2},  e^{\delta, n+1}\pm e^{\delta,n} \Bigr)_I   dW_s \Bigr]  \\
&\leq {\mathbb E} \Bigl[ \frac12\Vert e^{\delta,n+1} - e^{\delta,n}\Vert^2_{L^2(I)}\Bigr]
+ \frac{\epsilon^2}{2} \int_{t_n}^{t_{n+1}}  \Vert \partial_x[u_h^{\delta}(s) - u_h^{\delta,n} \pm u_h^\delta(t_n)]\Vert^2_{L^2(I)} \, ds\Bigr]
\nonumber \\
&\leq \frac{1}{2} {\mathbb E} \Bigl[ \Vert e^{\delta,n+1} - e^{\delta,n}\Vert^2_{L^2(I)}\Bigr]
+{\mathbb E} \Bigl[ \bigl( \frac{\epsilon^2}{2} + \frac{\delta}{2} \bigr)\tau\, 
\Vert \partial_x e^{\delta,n}\Vert^2_{L^2(I)} \nonumber\\ 
&\hskip 0.5in
+ \bigl(\frac{\epsilon^2}{2} + \frac{2}{\delta} \bigr)  \int_{t_n}^{t_{n+1}}  \Vert \partial_x[u_h^{\delta}(s) -  u_h^\delta(t_n)]\Vert^2_{L^2(I)} \, ds\Bigr]
\nonumber \\
&\leq \frac{1}{2} {\mathbb E} \Bigl[ \Vert e^{\delta,n+1} - e^{\delta,n}\Vert^2_{L^2(I)}\Bigr]
+\bigl( \frac{\epsilon^2}{2} + \frac{\delta}{2} \bigr) \tau \, {\mathbb E} \Bigl[  \Vert \partial_x e^{\delta,n}\Vert^2_{L^2(I)}\Bigr]\nonumber  \\
&\hskip 0.5in
+  C\bigl( 1+\delta^{-1}  \bigr) h^{-2} \, \mathbb{E} \Bigl[  \int_{t_n}^{t_{n+1}}  \Vert u_h^{\delta}(s) -  u_h^\delta(t_n)\Vert^2_{L^2(I)} \, ds\Bigr].\label{star_1}
\end{align}
An elementary calculation and an application of an inverse inequality yield
\begin{align}\nonumber
 &\bigl(\frac{\epsilon^2}2+\delta\bigr) \int_{t_n}^{t_{n+1}}
 \Bigl( \p_x u_h^\delta(s)- \p_x u_h^\delta(t_{n+1}) , \p_x e^{\delta,n+1} \Bigr)_I ds \\
 &\leq \frac{\delta}{8} \tau\, \Vert \partial_x e^{\delta,n+1}\Vert^2_{L^2(I)}
 + \frac{2 (\frac{\epsilon^2}{2} + \delta)^2}{\delta} \int_{t_n}^{t_{n+1}} \Vert \partial_x[u^\delta_h(s) - u^\delta_h(t_{n+1})]\Vert^2_{L^2(I)}\, ds\nonumber\\
 &\leq \frac{\delta}{8} \tau\, \Vert \partial_x e^{\delta,n+1}\Vert^2_{L^2(I)}
 +(\epsilon^2+2\delta)^2 \delta^{-1} h^{-2}  \int_{t_n}^{t_{n+1}}  \Vert u_h^{\delta}(s) -  u_h^\delta(t_n)\Vert^2_{L^2(I)} \, ds. \label{star_2}
\end{align}
By the monotonicity of $\arctan$ we get
\begin{align}
 &-\bigl( 1-\frac{\epsilon^2}2\bigr) \int_{t_n}^{t_{n+1}} \Bigl( \arctan \bigl(\p_x u_h^\delta(s) \bigr)
 \pm {\rm arctan} \bigl( \partial_x u_h^{\delta}(t_{n+1})\bigr) \nonumber \\
 &\hskip 1.5in -\arctan(\p_x u_h^{\delta,n+1}), \p_x e^{\delta, n+1} \Bigr)_I ds
\nonumber  \\
 &\leq -\bigl( 1-\frac{\epsilon^2}2\bigr) \int_{t_n}^{t_{n+1}} \Bigl( \arctan \bigl(\p_x u_h^\delta(s) \bigr)
 - {\rm arctan} \bigl( \partial_x u_h^{\delta}(t_{n+1})\bigr), \p_x e^{\delta,n+1} \Bigr) \, ds \nonumber\\
 &\leq  \frac{\delta}{8} \tau\, \Vert \partial_x e^{\delta,n+1}\Vert^2_{L^2(I)}
 + \frac{2}{\delta}(1-\frac{\epsilon^2}{2})^2 \int_{t_n}^{t_{n+1}} \Vert \partial_x[u^\delta_h(s) - u^\delta_h(t_{n+1})]\Vert^2_{L^2(I)}\, ds \nonumber\\
 &\leq \frac{\delta}{8} \tau\, \Vert \partial_x e^{\delta,n+1}\Vert^2_{L^2(I)}
 + (2-\epsilon^2)^2  \delta^{-1} h^{-2}  \int_{t_n}^{t_{n+1}}  \Vert u_h^{\delta}(s) -  u_h^\delta(t_n)\Vert^2_{L^2(I)} \, ds. \label{star_3}
\end{align}

Finally, substituting the above estimates into (\ref{e3.8}), 
summing over $n=0,1, 2, \cdots, N-1$, and using \eqref{prop3.7b} and
 the fact that $e^{\delta,0} = 0$ we get
\begin{align*}
&\sup_{0 \leq n \leq N} {\mathbb E}\Bigl[ \Vert e^{\delta,n}\Vert^2_{L^2(I)}\Bigr]  + \frac{\delta}{2} {\mathbb E} \Bigl[ \tau
\sum_{n=0}^N \Vert \partial_x e^{\delta,n+1}\Vert^2_{L^2(I)} \Bigr]  \\
&\hskip 0.5in
\leq C\bigl( 1+\delta^{-1}\bigr) h^{-2} \sum_{n=0}^N \int_{t_n}^{t_{n+1}} \mathbb{E}\Bigl[  \sup_{s\in [t_n, t_{n+1}]}  \Vert u_h^{\delta}(s) -  u_h^\delta(t_n)\Vert^2_{L^2(I)} \Bigr] \, ds\\
&\hskip 0.5in
 \leq  CT(1+\delta^{-1})^2 h^{-2} \tau,
\end{align*}
which infers \eqref{e3.5}. The proof is complete.
\end{proof}

We conclude this section by stating the following error estimates  
for the fully discrete finite element solution $u_h^{\delta,n}$ as an 
approximation to the solution of the original mean curvature flow 
equation \eqref{e1.3}.  

\begin{theorem}\label{thm3.3}
Let $u$ and $u_h^{\delta, n}$ denote respectively the solutions 
of SPDE \eqref{e1.3} and scheme \eqref{e3.1}--\eqref{e3.1a}. Under assumptions 
of Theorems \ref{thm2.1}, \ref{thm3.1}, and  \ref{thm3.2},  there 
holds the following error estimate: 
\begin{align}\label{e3.30}
&\sup_{0\leq n\leq N}  \mathbb{E}  \Bigl[\bigl\| u(t_n) 
- u_h^{\delta,n} \bigr\|_{L^2(I)}^2 \Bigr]
+\delta\, \mathbb{E} \Bigl[\sum_{n=0}^N \tau \bigl\|\p_x u(t_n) 
-\p_x u_h^{\delta,n} \bigr\|_{L^2(I)}^2 \Bigr] \\
&\qquad
\leq CT\delta + C\bigl(1+\delta^{-2} \bigr) h^2
+ CT\bigl(1+ \delta^{-2}  \bigr) h^{-2} \tau.  \nonumber
\end{align}
\end{theorem}

Inequality \eqref{e3.30} follows immediately from Theorems \ref{thm2.3}, 
\ref{thm3.1} and  \ref{thm3.2}, and an application of the triangle inequality.

\begin{remark}\label{addcoup}
We note that the main reason to have a restrictive 
coupling between numerical parameters in \eqref{e3.30} is due to the lack of
H\"older continuity (in time) estimate for $\p_x u_h^\delta$ in $L^2$-norm.  
On the other hand, it can be shown that, under a stronger regularity 
assumption, the estimate \eqref{e3.30} can be improved to
\begin{align}\label{e3.30a}
\sup_{0\leq n\leq N}  \mathbb{E}  \Bigl[\bigl\| u(t_n) 
&- u_h^{\delta,n} \bigr\|_{L^2(I)}^2 \Bigr]\\
&+ \delta \, \mathbb{E} \Bigl[\sum_{n=0}^N \tau \bigl\|\p_x u(t_n) 
-\p_x u_h^{\delta,n} \bigr\|_{L^2(I)}^2 \Bigr] \leq C\bigl( h^2 + \tau\bigr)\, .
\nonumber
\end{align}
This is because we no longer need to use the inverse inequality to 
get \eqref{star_1}--\eqref{star_3}, and \eqref{e3.30a} can be obtained by
starting with a control of the time discretization first. 
\end{remark}

\section{Numerical experiments} \label{sec-4}
In this section we shall first present some numerical experiments
to gauge the performance of the proposed fully discrete finite
element method and to examine the effect of the noise for long-time dynamics
of the stochastic MCF of planar graphs, and we then present a numerical
study of the stochastic MCF driven by both colored and 
space-time white noises where no theoretical result is known so
far in the literature. 

\subsection{Verifying the rate of convergence of time discretization} 
\label{sec-4.1}
To verify the rate of convergence of the time discretization obtained in 
Theorem \ref{thm3.3}, in this first test we use the following 
parameters $\epsilon=1$, $\delta=10^{-5}$, and $T=0.1$. 
In order to computationally generate a driving reference 
${\mathbb R}$-valued Wiener process, we use the smaller time
step $\tau=10^{-5}$. The initial condition is set to be $u_0(x)=\sin(\pi x)$.
To calculate the rate, we compute the solution $u_h^{\delta,n}$
for varying $\tau=0.0005, 0.001, 0.002, 0.004$. We take $500$ stochastic
samples at each time step $t_n$ in order to compute the expected values 
of the $L^{\infty}(L^2)$-norm of the error.
The computed errors along with the computed convergence rates
are exhibited in Table \ref{tab1}. The numerical results 
confirm the theoretical result of Theorem \ref{thm3.2}.

\begin{table}[htbp]
\begin{center}
\begin{tabular}{|c|c|c|}
\hline
  & Expected values of error & order of convergence\\ \hline
dt=0.004 & 0.41965657 & $-$ \\ \hline
dt=0.002 & 0.27206448 & 0.62526\\ \hline
dt=0.001 & 0.18136210 & 0.58508\\ \hline
dt=0.0005 & 0.12373884 & 0.55157\\ \hline
\end{tabular}
\medskip
\caption{Computed time discretization errors and convergence rates.}
\label{tab1}
\end{center}
\end{table}

\subsection{Dynamics of the stochastic MCF} \label{sec-4.3}
We shall perform several numerical tests to demonstrate 
the dynamics of the stochastic MCF with different magnitudes of noise 
(i.e., different sizes of the parameter $\epsilon$).

Figure \ref{fig5} shows the surface plots of the computed solution
$u_h^{\delta,n}$ at one stochastic sample over the space-time domains 
$(0,1)\times (0,0.1)$ (left) and $(0,1)\times (0, 2^8\times 10^{-5})$ 
(right) with the initial value $u_0(x)=\sin(\pi x)$ and the noise intensity
parameter $\epsilon=0.1$. The test shows that the solution converges to a steady 
state solution at the end.

\begin{figure}[htb]
\centerline{
\includegraphics[height=1.7in,width=4.0in]{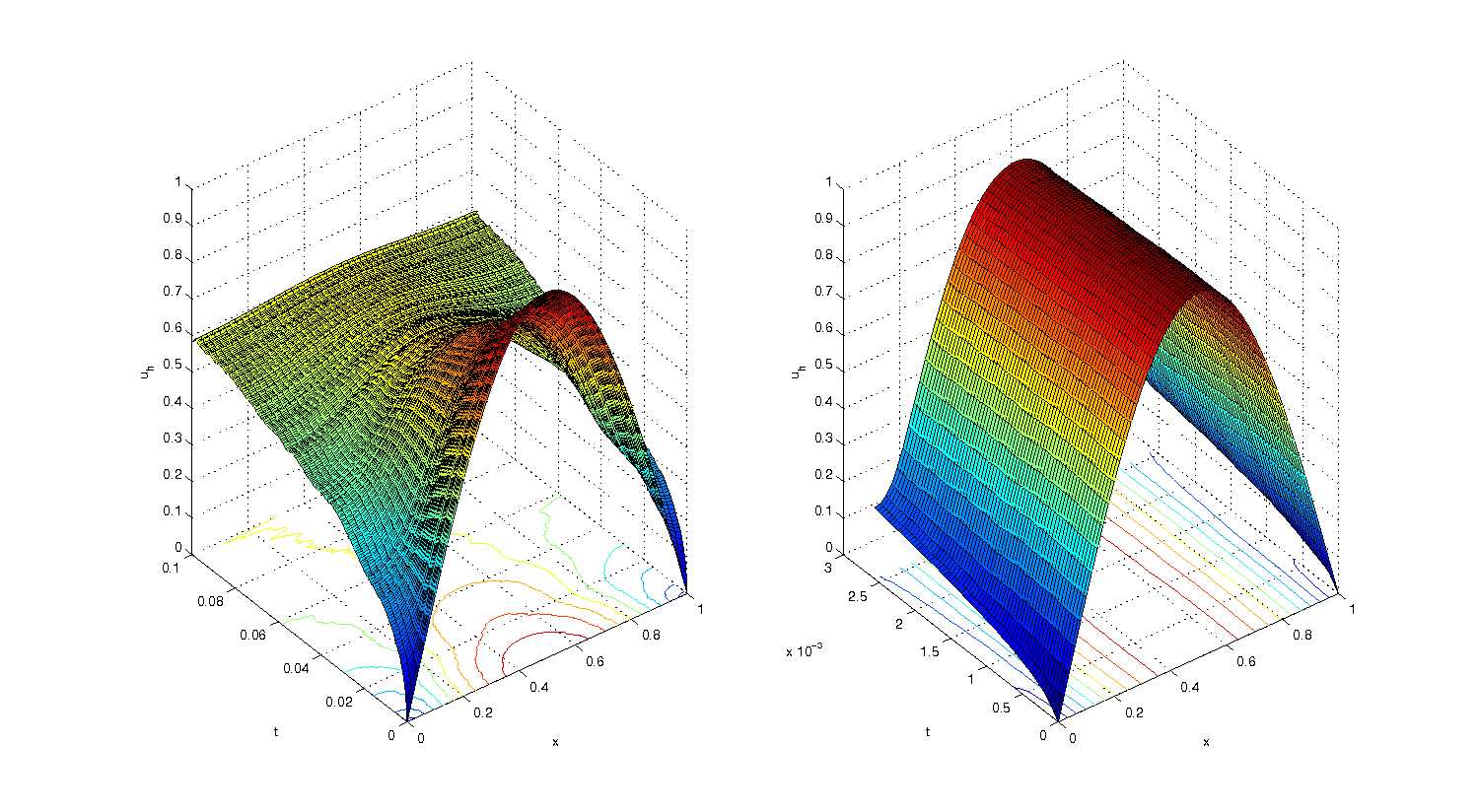}%
}
\figcaption{Surface plots of computed solution at a fixed stochastic
sample on the space time domains $(0,1)\times (0,0.1)$ (left) and 
$(0,1)\times (0, 2^8\times 10^{-5})$ (right). $u_0(x)=\sin(\pi x)$
and $\epsilon=0.1$.}\label{fig5}
\end{figure}

Figures \ref{fig6}--\ref{fig8} are the counterparts of 
Figure \ref{fig5} with noise intensity parameter $\epsilon=1, \sqrt{2}, 5$, 
respectively. We note that the error estimate of
Theorem \ref{thm3.3} does not apply to the latter case 
because the condition $\epsilon \leq \sqrt{2(1+ \delta)}$ is violated.
However, the computation result suggests that the 
stochastic MCF also converges to the steady state 
solution at the end although the paths to reach the steady 
state are different for different noise intensity parameter $\epsilon$.

\begin{figure}[htb]
\centerline{
\includegraphics[height=1.7in,width=4.0in]{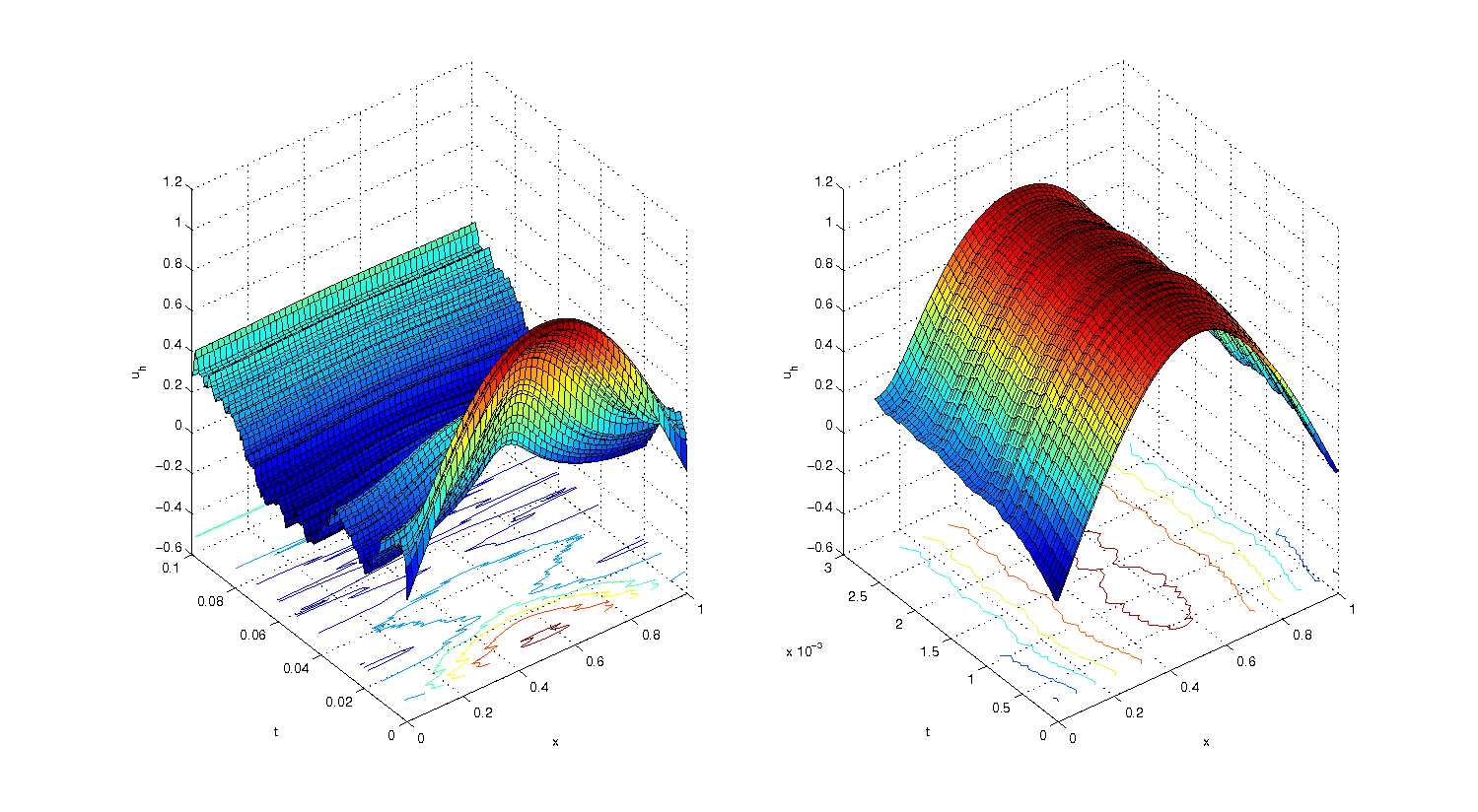}%
}
\figcaption{Surface plots of computed solution at a fixed stochastic
sample on the space time domains $(0,1)\times (0,0.1)$ (left) and
$(0,1)\times (0, 2^8\times 10^{-5})$ (right).  $u_0(x)=\sin(\pi x)$
and $\epsilon=1$.} \label{fig6}
\end{figure}

\begin{figure}[htb]
\centerline{
\includegraphics[height=1.7in,width=4.0in]{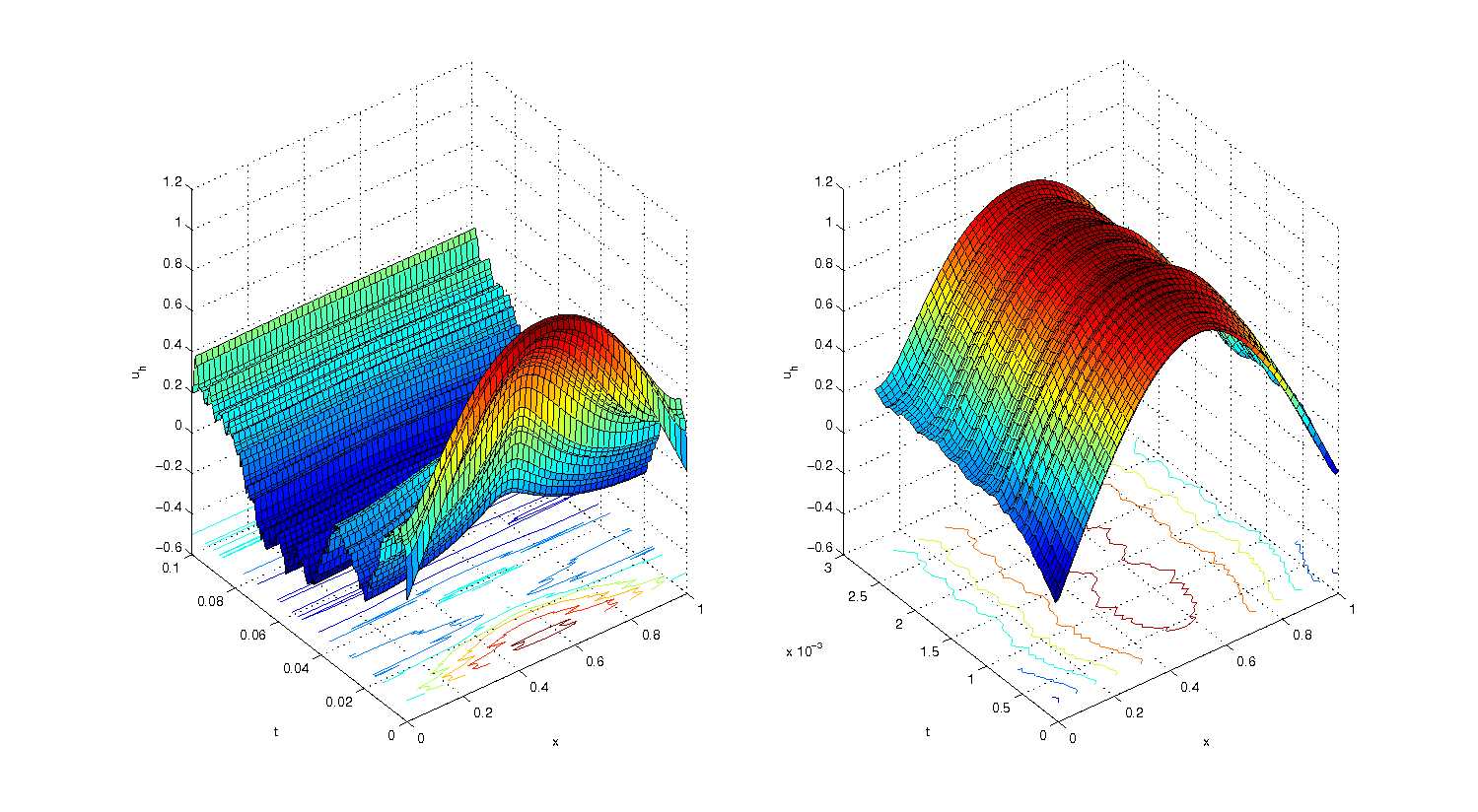}%
}
\figcaption{Surface plots of computed solution at a fixed stochastic
sample on the space time domains $(0,1)\times (0,0.1)$ (left) and
$(0,1)\times (0, 2^8\times 10^{-5})$ (right).  $u_0(x)=\sin(\pi x)$
and $\epsilon=\sqrt{2}$.} \label{fig7}
\end{figure}

\begin{figure}[htb]
\centerline{
\includegraphics[height=1.7in,width=4.0in]{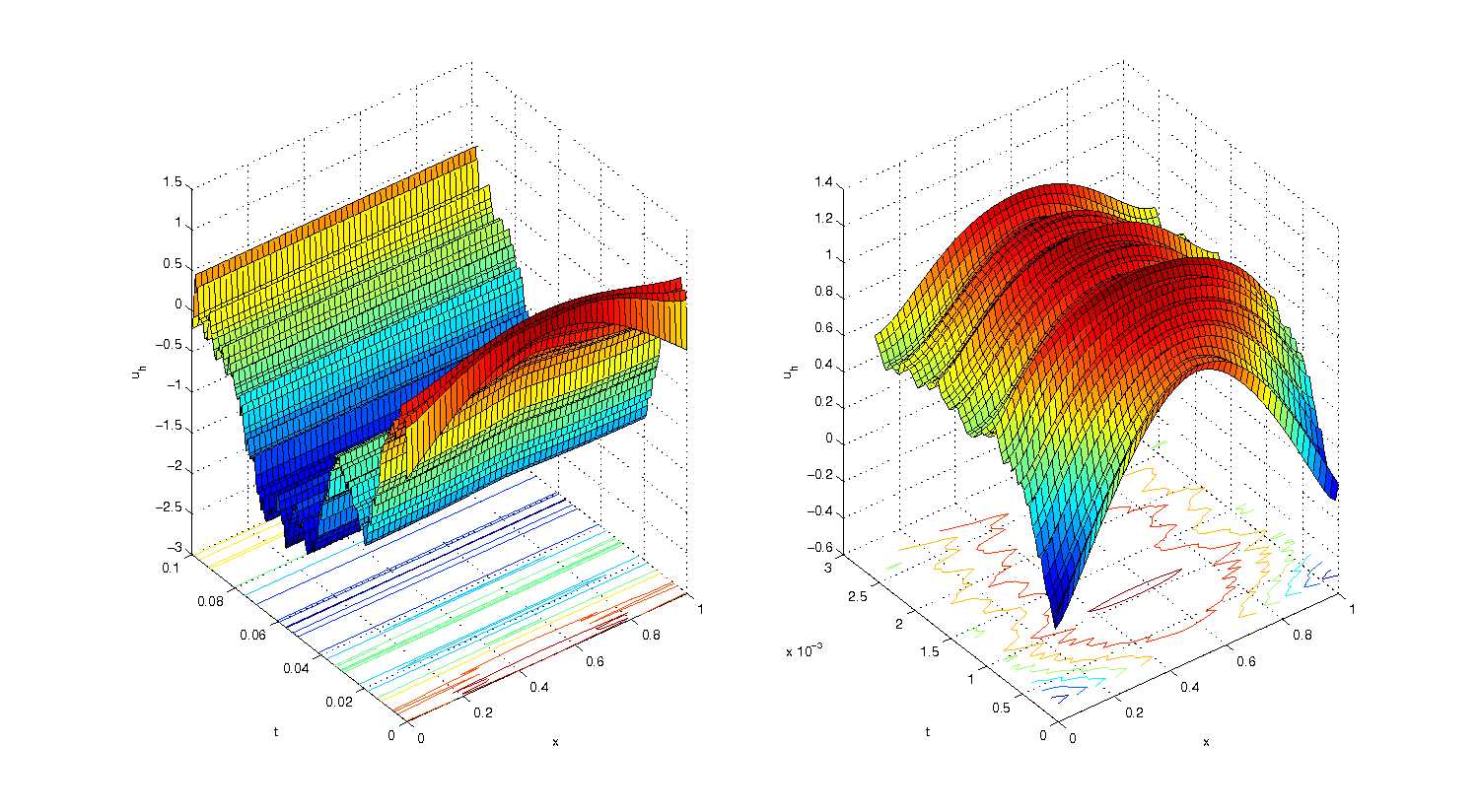}%
}
\figcaption{Surface plots of computed solution at a fixed stochastic
sample on the space time domains $(0,1)\times (0,0.1)$ (left) and
$(0,1)\times (0, 2^8\times 10^{-5})$ (right).  $u_0(x)=\sin(\pi x)$
and $\epsilon=5$.} \label{fig8}
\end{figure}

We then repeat the above four tests after replacing the smooth 
initial function $u_0$ by the following non-smooth initial function:

\begin{equation}\label{e4.2}
 u_0(x)=\begin{cases}
 10x, & \text{if}\ x\leq0.25, \\
 5-10x, & \text{if}\ 0.25<x\leq0.5, \\
 10x-5, & \text{if}\ 0.5<x\leq0.75, \\
 10-10x, & \text{if}\ 0.75<x\leq 1. \\
 \end{cases}
\end{equation}
The surface plots of the computed solutions are shown in Figures 
\ref{fig9}--\ref{fig12}, respectively.  Again, the numerical results
suggest that the solution of the stochastic MCF converges to the steady state
solution at the end although the paths to reach the steady state 
are different for different noise intensity parameter $\epsilon$.
As expected, the geometric evolution dominates for small $\epsilon$,
but the noise dominates the geometric evolution for large $\epsilon$. 

\begin{figure}[htb]
\centerline{
\includegraphics[height=1.7in,width=4.0in]{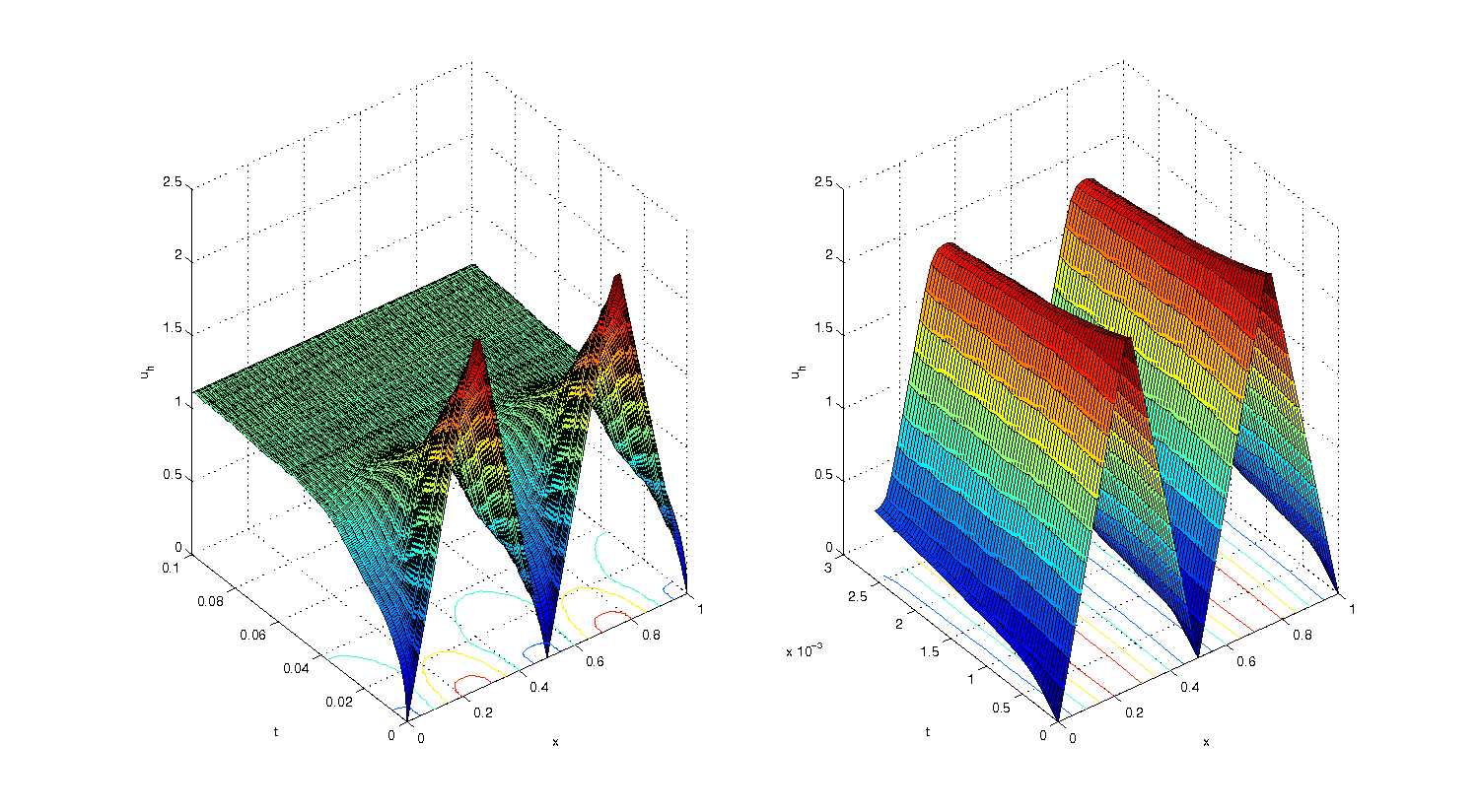}%
}
\figcaption{Surface plots of computed solution at a fixed stochastic
sample on the space time domains $(0,1)\times (0,0.1)$ (left) and
$(0,1)\times (0, 2^8\times 10^{-5})$ (right).  $u_0$ is given in
\eqref{e4.2} and $\epsilon=0.1$. }\label{fig9} 
\end{figure}
\begin{figure}[htb]
\centerline{
\includegraphics[height=1.7in,width=4.0in]{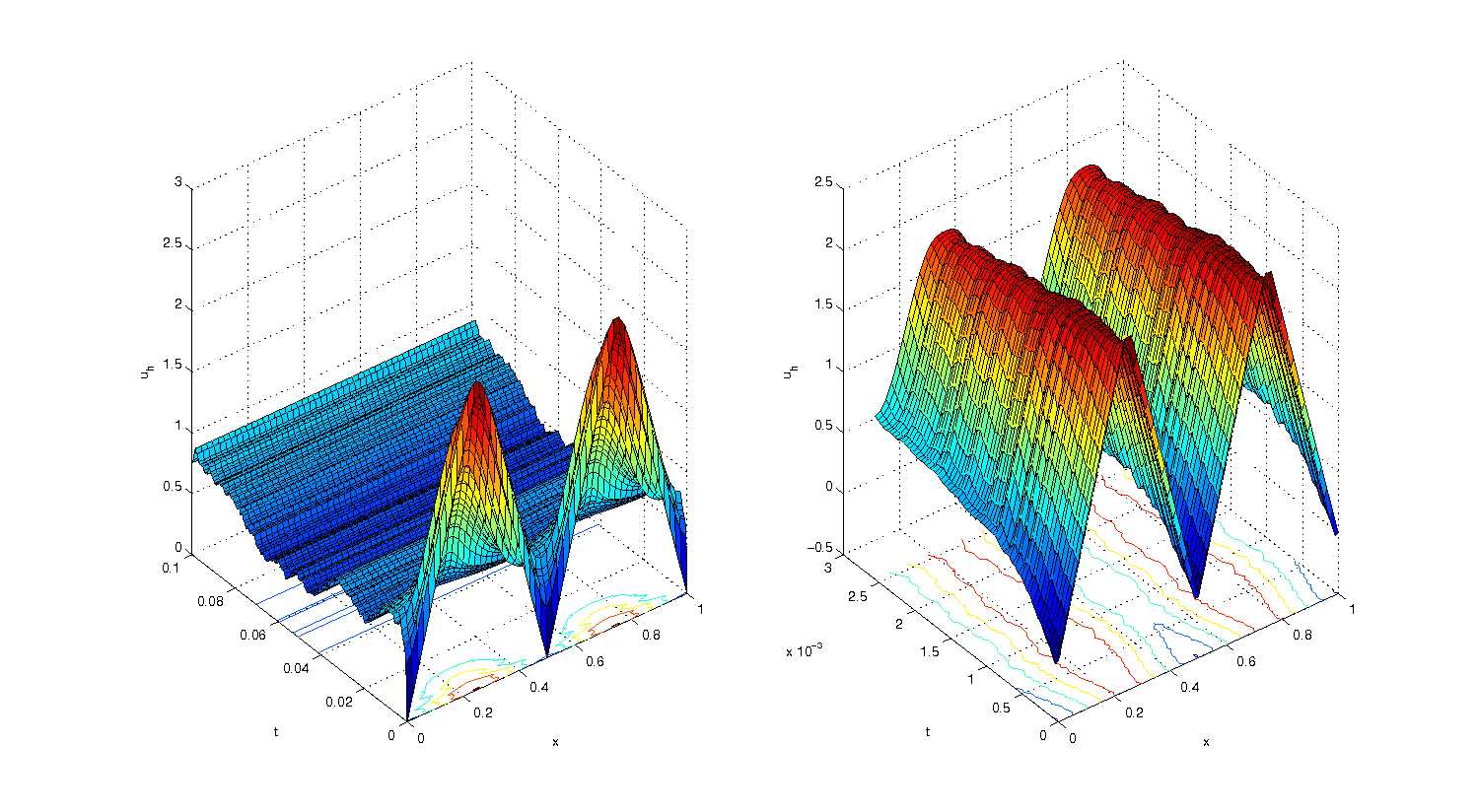}%
}
\figcaption{Surface plots of computed solution at a fixed stochastic
sample on the space time domains $(0,1)\times (0,0.1)$ (left) and
$(0,1)\times (0, 2^8\times 10^{-5})$ (right).  $u_0$ is given in
\eqref{e4.2} and $\epsilon=1$.}\label{fig10}
\end{figure}
\begin{figure}[htb]
\centerline{
\includegraphics[height=1.7in,width=4.0in]{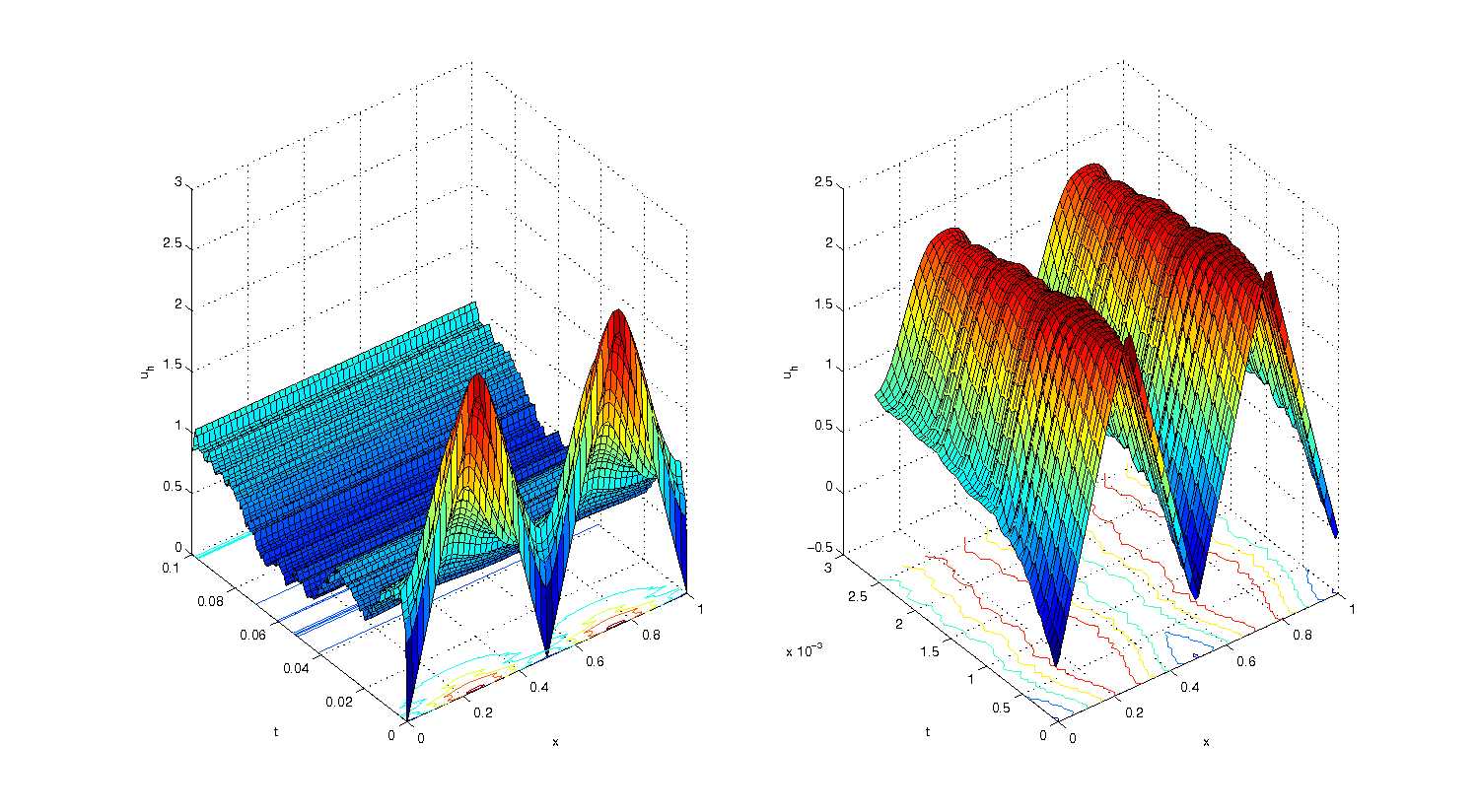}%
}
\figcaption{Surface plots of computed solution at a fixed stochastic
sample on the space time domains $(0,1)\times (0,0.1)$ (left) and
$(0,1)\times (0, 2^8\times 10^{-5})$ (right).  $u_0$ is given in
\eqref{e4.2} and $\epsilon=\sqrt{2}$.} \label{fig11}
\end{figure}
\begin{figure}[htb]
\centerline{
\includegraphics[height=1.6in,width=4.0in]{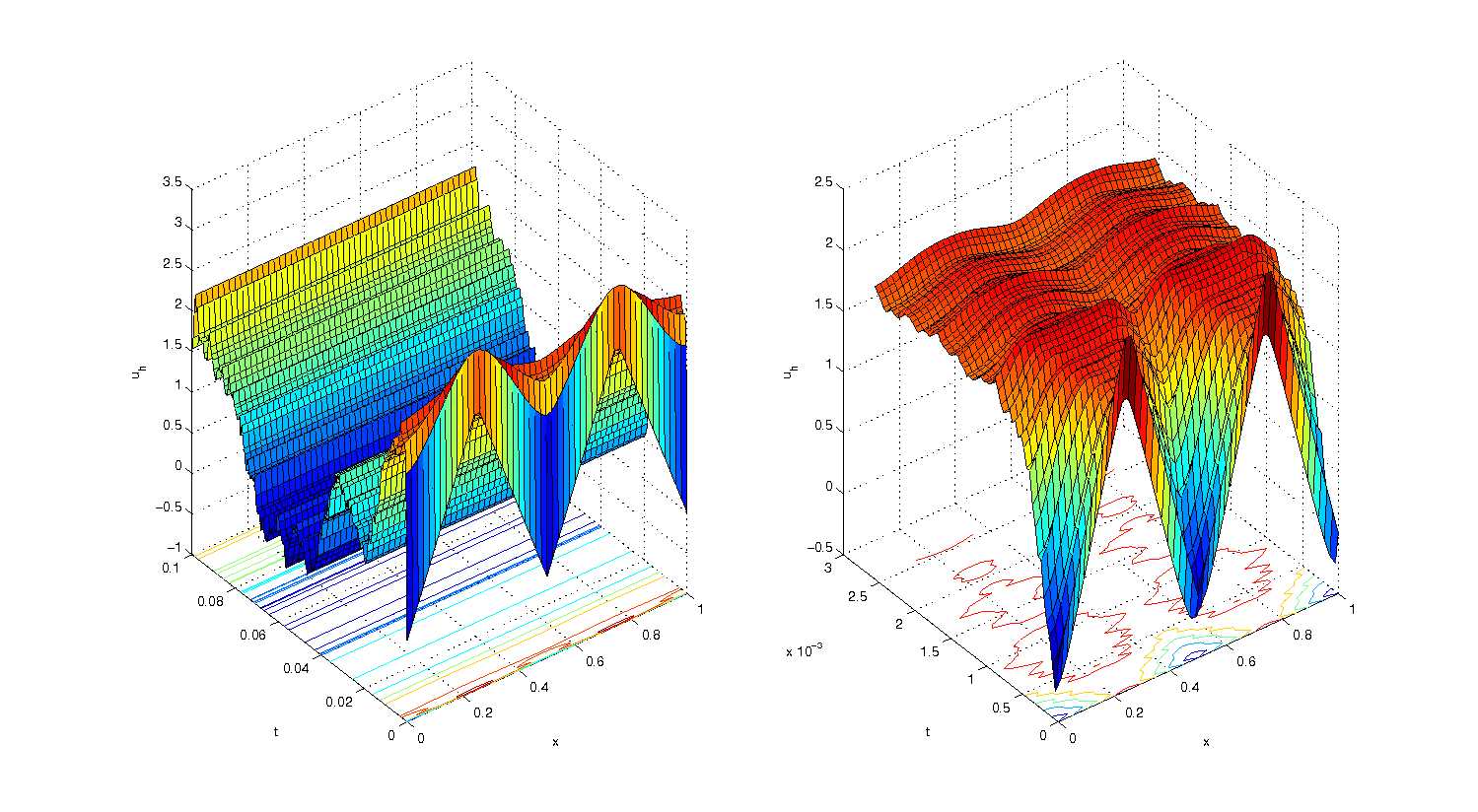}%
}
\figcaption{Surface plots of computed solution at a fixed stochastic
sample on the space time domains $(0,1)\times (0,0.1)$ (left) and
$(0,1)\times (0, 2^8\times 10^{-5})$ (right).  $u_0$ is given in
\eqref{e4.2} and $\epsilon=5$.} \label{fig12}
\end{figure}

\subsection{Verifying energy dissipation}\label{sec-4.5}
It follows from \eqref{e2.6} that the ``energy" 
$J(t):=\frac12 \mathbb{E}\bigl[\|\partial_x u^{\delta}(t)\|_{L^2(I)}^2\bigr]$
decreases monotonically  in time. In the following we verify this fact 
numerically. Again, we consider the case with the initial 
function $u_0(x)=\sin(\pi x)$ and the noise intensity parameter $\epsilon=1$. 
It is not hard to prove that $J(t)$ converges to zero as $t\to \infty$.
Figure \ref{fig13} plots the computed $J(t)$ as a function of $t$. 
The numerical result suggests that $J(t)$ does not change
anymore for $t\geq 0.1$.  

\begin{figure}[tbh]
\centerline{
\includegraphics[height=1.9in,width=2.5in]{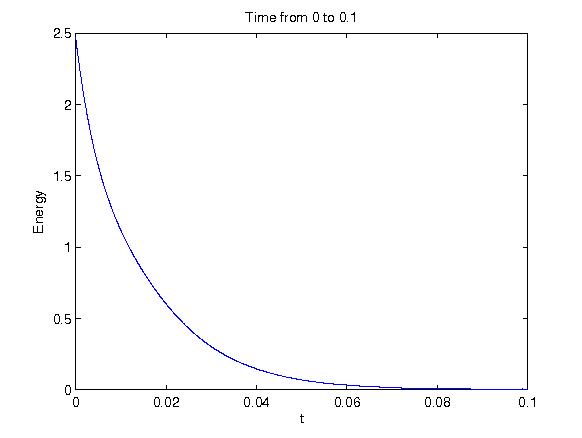}%
}
\figcaption{Decay of the energy $J(t)$ on the interval $(0, 0.1)$. } 
\label{fig13}
\end{figure}


\subsection{Thresholding for colored noise}\label{sec-4.1a}
In this subsection we present a computational study 
of the interplay of noise and geometric evolution in \eqref{e1.3}, 
which is beyond our theoretical results in section \ref{sec-3.1} and
\ref{sec-3.2}. For this purpose, we use driving colored noise 
represented by the $Q$-Wiener process ($J \in {\mathbb N}$)
\begin{equation}\label{wn1}
W_t = \sum_{j=1}^{J} q_j^{\frac12} \beta_j(t) e_j\,,
\end{equation}
where $\{ \beta_j(t);\, t \geq 0\}_{j\geq 1}$  denotes a family of real-valued
independent Wiener processes on $\bigl( \Omega, {\mathcal F},
{\mathbb F}, {\mathbb P}\bigr)$, and $\{ (q_j, e_j)\}_{j=1}^{J}$ is an 
eigen-system of the symmetric, non-negative trace-class operator 
$Q: L^2(I) \rightarrow L^2(I)$, with $e_j = \sqrt{2} \sin(j \pi x)$. 
In particular, we like to numerically address the following questions:
\begin{enumerate}
\item[(A)] {\em Thresholding}: By Theorem~\ref{thm2.1}, strong solutions 
of \eqref{e1.3} exist for $\varepsilon \leq \sqrt{2}$, and a similar 
result can be shown for the PDE problem with the noise \eqref{wn1}. What are 
admissible intensities of the noise suggested by computations?  
Moreover, what do the computations suggest about the stochastic MCF
in the case of spatially white noise (i.e., $q_j\equiv 1, J=\infty$)
where no theoretical result is available so far?  
\item[(B)] {\em General initial profiles}: The deterministic evolution of 
Lipschitz initial graphs is well-understood. For example, the (upper) graph 
of two touching spheres may trigger non-uniqueness. What are the 
regularization and the noise excitation effects in the case
of the initial data with infinite energy and using different noises? 
\end{enumerate}

Recall that the estimate in Proposition~\ref{prop3.3} for $V_r^h$-valued 
solution $u^{\delta,n}_h$ suggests that $\varepsilon >0$ ought be
sufficiently small to ensure the existence. In our test, we 
employ the colored noise \eqref{wn1} with $q_j^{\frac12} = j^{-0.6}$, $J = 20$,
and the following non-Lipschitz initial data:
\begin{equation}\label{app1} 
u_0(x) = |0.5 - x |^\kappa \qquad \forall\, x \in (0,1)\,,
\end{equation}
where $\kappa = 0.1$. In addition, we set $(\tau, h) = (0.01, 0.02)$
and $T = \frac12$. Figure~\ref{figu1} shows the single trajectory
of the stochastic MCF plotted as graphs over the space-time domain with, 
respectively, $\epsilon=0.1, 0.5, \sqrt{2}$. The results indicate thresholding, 
namely, the trajectories grow rapidly in time for sufficiently large
values $\varepsilon$, and the noise effect dominates the geometric evolution.
\begin{figure}[htb]
\begin{center}
\includegraphics[height=1.7in,width=2.0in]{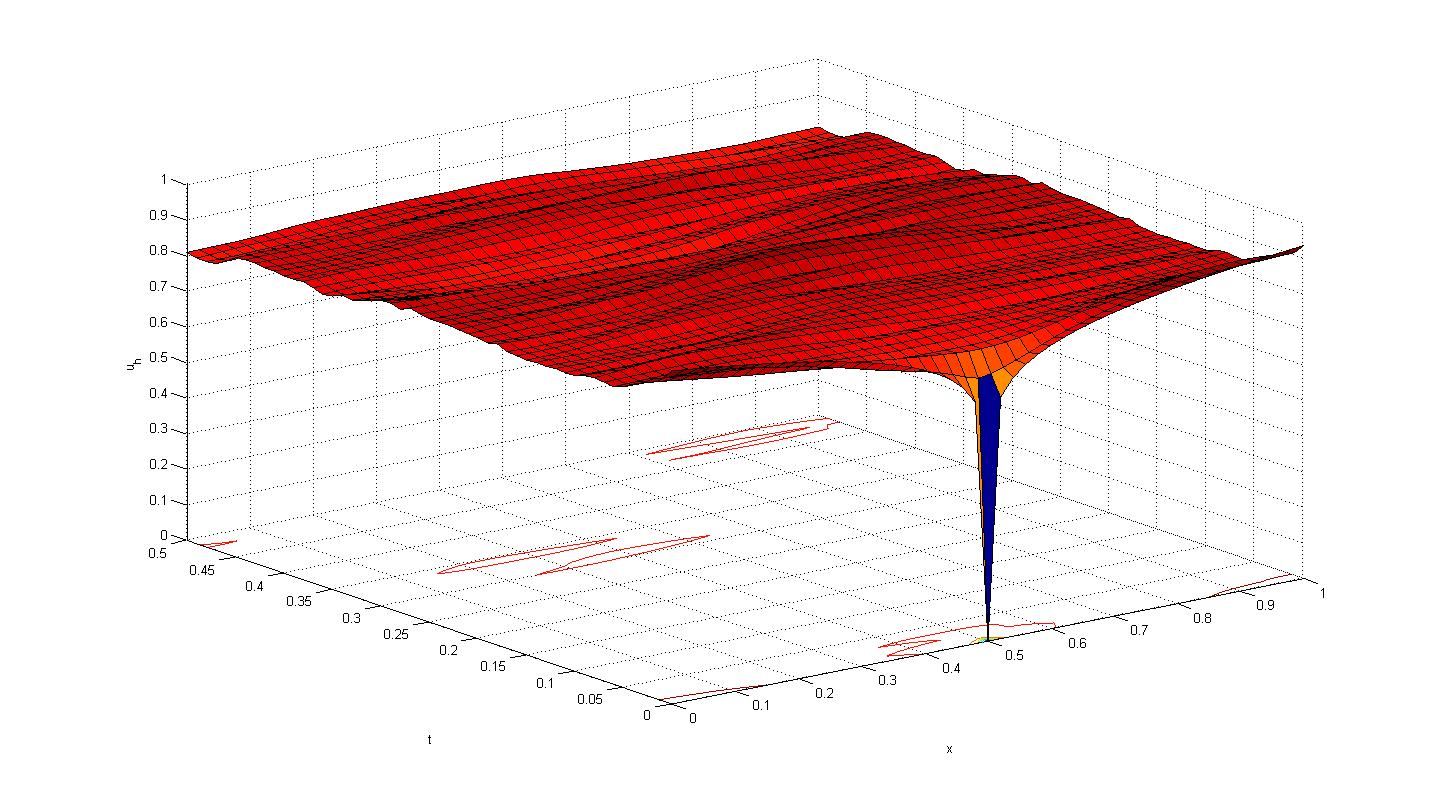}
\includegraphics[height=1.7in,width=2.0in]{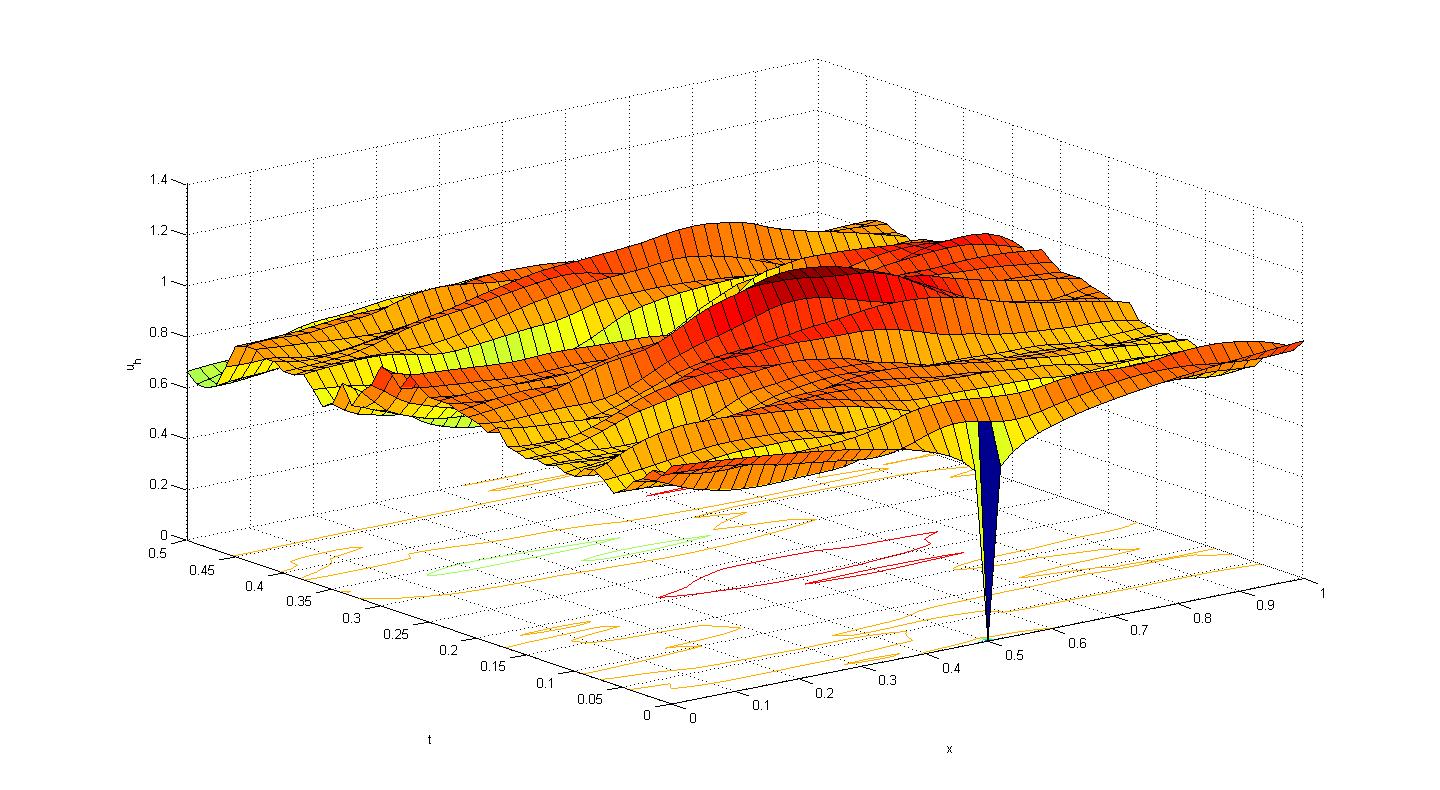}\\
\includegraphics[height=1.7in,width=2.0in]{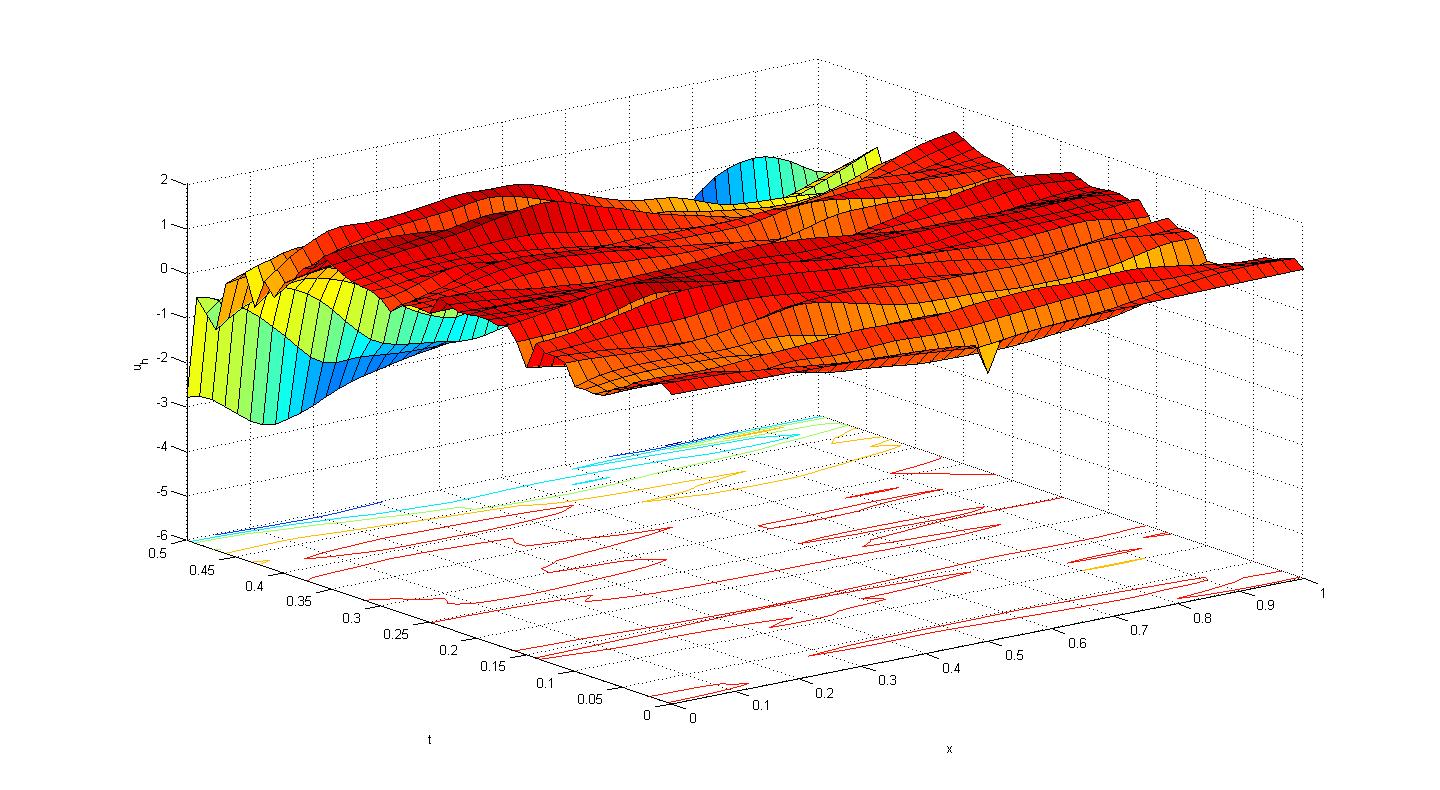}
\caption{Thresholding for colored noise: Trajectories for
$\varepsilon = 0.1$ (top left), 
$\varepsilon = 0.5$ (top right), $\varepsilon = \sqrt{2}$ (bottom).} \label{figu1}
\end{center}
\end{figure}
The excitation effect of the noise on the geometric evolution is illustrated 
by corresponding plots for the evolution of the functional 
$n \mapsto \Vert \p_x u_h^{\delta,n}(\omega)\Vert^2_{L^2}$ vs its expectation 
$n \mapsto {\mathbb E}\bigl[ \Vert \p_x u_h^{\delta,n}\Vert^2_{L^2}\bigr]$
in Figure~\ref{figu1a} and \ref{figu1b}. We observe that the geometric 
evolution dominates for small values of $\varepsilon$, while the noise 
evolution takes over for large values of $\varepsilon$.
\begin{figure}[htb]
\begin{center}
\includegraphics[height=3.0in,width=4.0in]{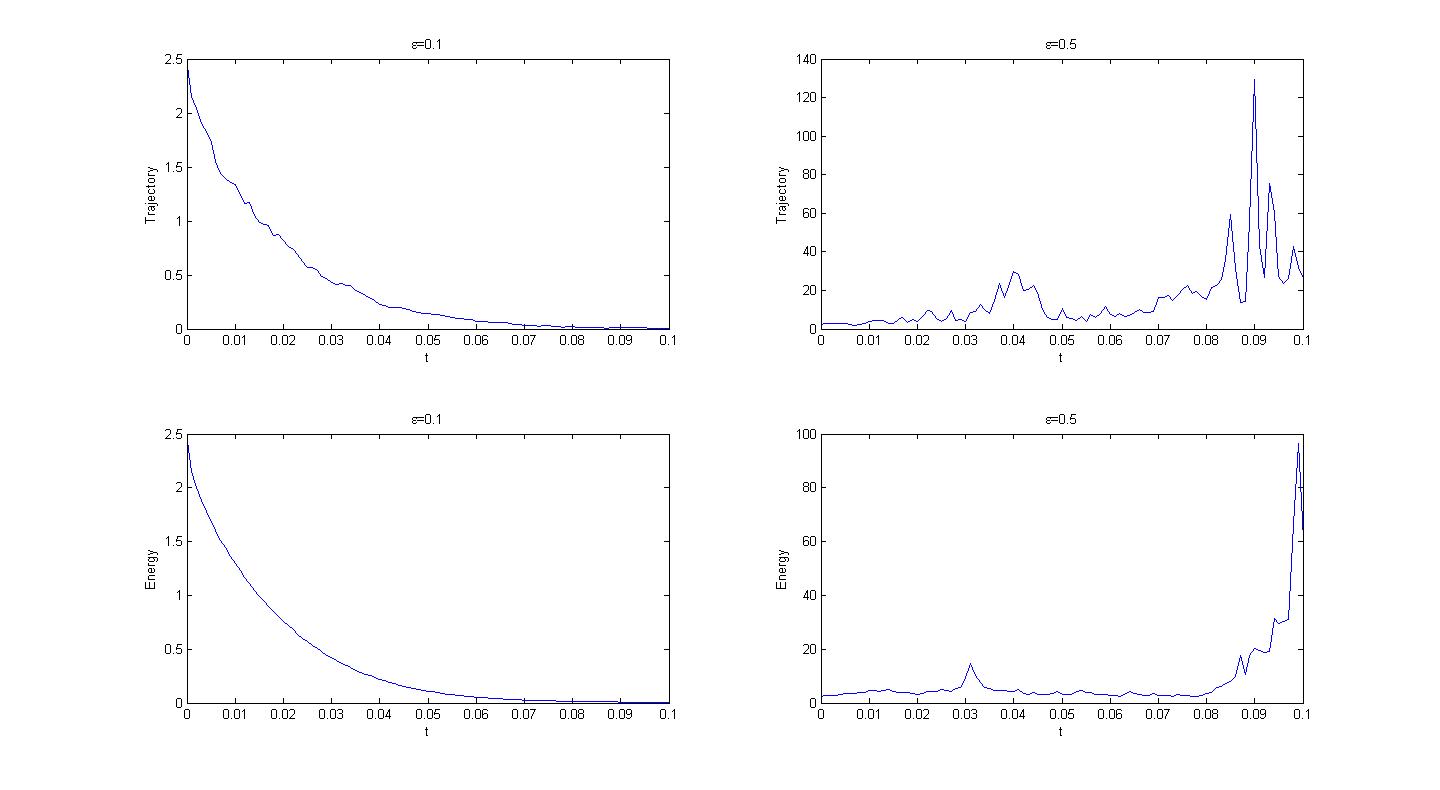}
\caption{Geometric evolution vs colored noise evolution ($q^{\frac12}_j 
= j^{-0.6}$, $J=20$): 1st row: single trajectory for 
$n \mapsto \Vert \p_x u_h^{\delta,n}(\omega) \Vert^2_{L^2}$ 
and $\varepsilon = 0.1$ (left), $\varepsilon = 0.5$ (right); 2nd 
row: $n \mapsto {\mathbb E}\bigl[ \Vert \p_x u_h^{\delta,n}\Vert^2_{L^2}\bigr]$ 
for  $\varepsilon = 0.1$ (left), $\varepsilon = 0.5$ (right).} \label{figu1a}
\end{center}
\end{figure}

\begin{figure}[htb]
\begin{center}
\includegraphics[height=3.0in,width=4.0in]{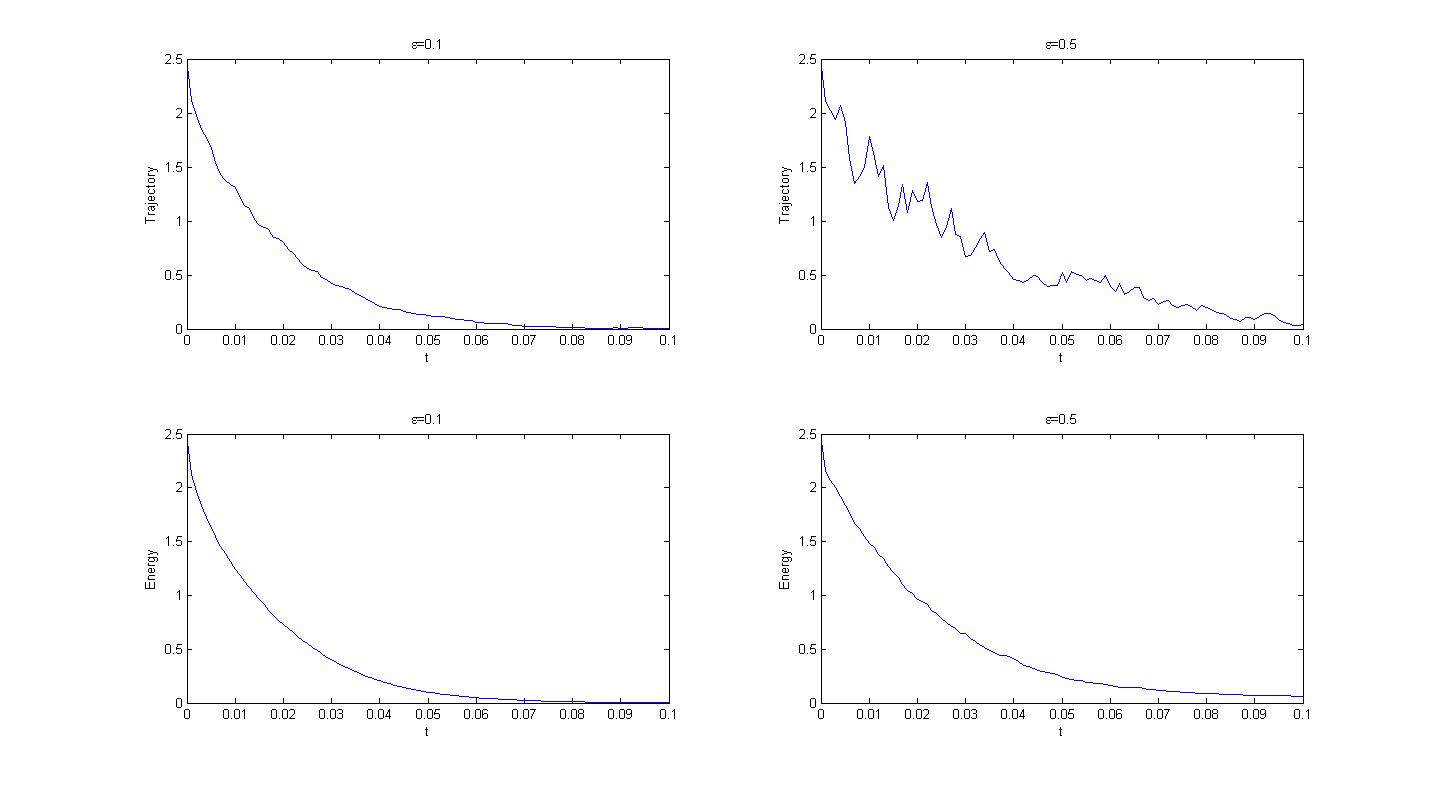}
\caption{Geometric evolution vs colored noise evolution
($q^{\frac12}_j = j^{-1}$, $J=50$): 1st row: single trajectory for
$n \mapsto \Vert \p_x u_h^{\delta,n}(\omega) \Vert^2_{L^2}$ 
and $\varepsilon = 0.1$ (left), $\varepsilon = 0.5$ (right); 2nd row:
$n \mapsto {\mathbb E}\bigl[ \Vert \p_x u_h^{\delta,n}\Vert^2_{L^2}\bigr]$ 
for  $\varepsilon = 0.1$ (left), $\varepsilon = 0.5$ (right).} \label{figu1b}
\end{center}
\end{figure}

\subsection{Thresholding for white noise}\label{sec-4.1b}
We now consider the case of white noise in \eqref{e3.1}--\eqref{e3.1a}, 
that is, $q_j \equiv 1$ in \eqref{wn1} and $J = \infty$, for which the 
solvability of \eqref{e1.2a} is not known. 
Figure~\ref{figu2} shows the single trajectory of the stochastic MCF 
(with the same data as in section~\ref{sec-4.1a}) 
plotted as graphs over the space-time domain with, respectively,
$\epsilon=0.1, 0.5, \sqrt{2}$.  We observe a very rapid
growth of trajectories (numerical values range between $10^{14}$ and
$10^{21}$) even for small values of $\varepsilon >0$. 
These numerical results suggest either a rapid growth or a finite time 
explosion for the stochastic MCF in the case of white noise.

\begin{figure}[htb]
\begin{center}
\includegraphics[height=1.7in,width=2.0in]{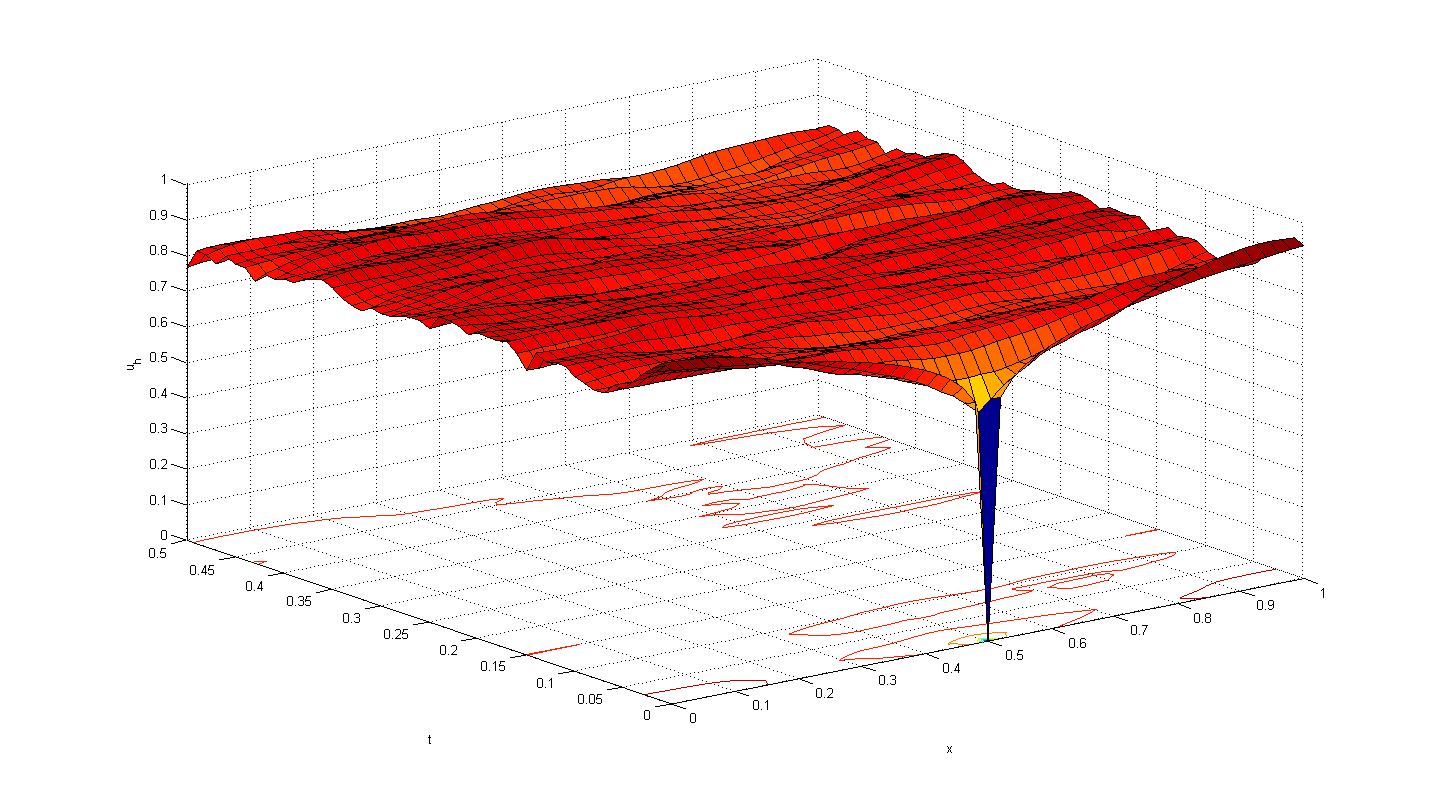}
\includegraphics[height=1.7in,width=2.0in]{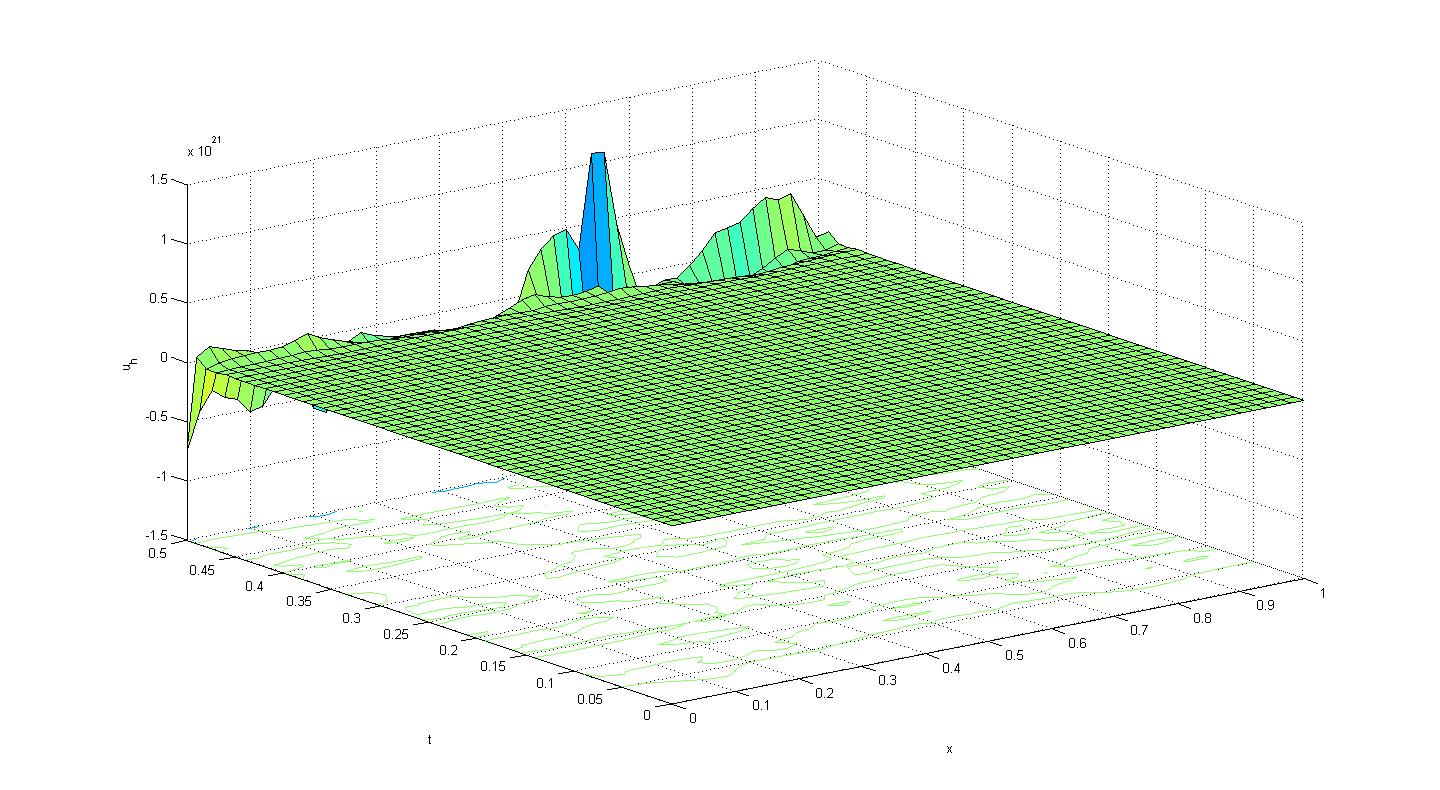} \\
\includegraphics[height=1.7in,width=2.0in]{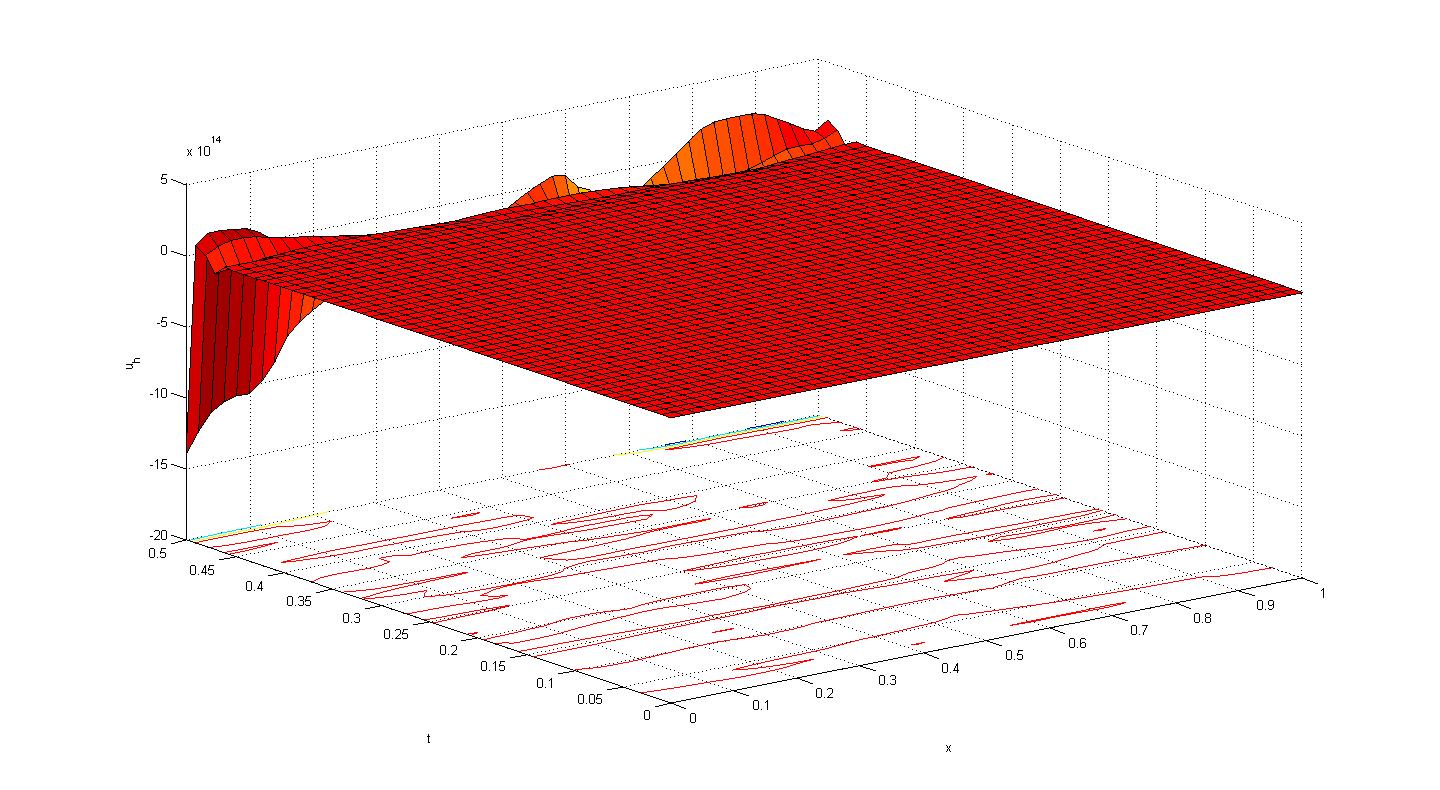}
\caption{Thresholding and white noise: $\varepsilon = 0.1$ (top left), 
$\varepsilon = 0.5$ (top right) $\varepsilon = \sqrt{2}$ (bottom).} \label{figu2}
\end{center}
\end{figure}


\end{document}